\documentclass[12pt]{article}
\usepackage{amsmath,amsfonts,amssymb,amsthm}

\newcommand{\I}{{\bf 1}}
\newtheorem{proposition}{Proposition}[section]
\newtheorem{theorem}[proposition]{Theorem}
\newtheorem{corollary}[proposition]{Corollary}
\newtheorem{lemma}[proposition]{Lemma}
\newtheorem{remark}[proposition]{Remark}

\newtheorem{example}[proposition]{Example}

\numberwithin{equation}{section}
\newcommand{\nc}{\newcommand}
\nc{\R}{{\mathbb R}}
\nc{\bS}{{\mathbb S}^{d-1}}
\nc{\N}{{\mathbb N}}
\nc{\Z}{{\mathbb Z}}
\nc{\BP}{\mathbb{P}}
\nc{\BE}{\mathbb{E}}
\nc{\BQ}{\mathbb{Q}}
\nc{\bN}{{\mathbf N}}
\nc{\BX}{{\mathbb X}}
\nc{\BY}{{\mathbb Y}}
\nc{\cB}{{\mathcal B}}
\nc{\cX}{{\mathcal X}}
\nc{\cY}{{\mathcal Y}}
\nc{\dint}{{\rm d}}
\DeclareMathOperator{\BC}{{\mathbb Cov}}
\DeclareMathOperator{\BV}{\operatorname{Var}}

\DeclareMathOperator{\dom}{dom}

\DeclareMathOperator{\interior}{int}

\begin{document}
\title{Normal approximation on Poisson spaces:\\ Mehler's formula, second order Poincar\'e inequalities and stabilization}
\date{\today}

\renewcommand{\thefootnote}{\fnsymbol{footnote}}

\author{G\"unter Last\footnotemark[1]\,,
Giovanni Peccati\footnotemark[2]\, and Matthias Schulte\footnotemark[1]}

\footnotetext[1]{Institute of Stochastics, Karlsruhe Institute of Technology,
Germany, guenter.last@kit.edu and matthias.schulte@kit.edu. GL and MS were supported by the German Research Foundation
(DFG) through the research unit "Geometry and Physics of Spatial Random Systems" under the grant HU 1874/3-1.}

\footnotetext[2]{Mathematics Research Unit, Luxembourg University, Luxembourg, giovanni.peccati@gmail.com. GP was partially supported by the grant F1R-MTH-PUL-12PAMP  (PAMPAS), from Luxembourg University.}

\maketitle
\begin{abstract}

We prove a new class of inequalities, yielding bounds for the normal approximation in the Wasserstein and the
Kolmogorov distance of functionals of a general Poisson process (Poisson random measure). Our
approach is based on an iteration of the classical Poincar\'e inequality, as
well as on the use of Malliavin operators, of Stein's method, and of an
(integrated) Mehler's formula, providing a representation of the
Ornstein-Uhlenbeck semigroup in terms of thinned Poisson processes. Our
estimates only involve first and second order differential operators, and have
consequently a clear geometric interpretation.
In particular we will show that our results are
perfectly tailored to deal with the normal approximation of geometric functionals displaying a weak form of stabilization, and with non-linear functionals of Poisson shot-noise processes.
We discuss two examples of stabilizing functionals  in
great detail: (i) the edge length of the $k$-nearest neighbour graph, (ii)
intrinsic volumes of $k$-faces of Voronoi tessellations.
In all these examples we obtain rates of convergence (in the Kolmogorov and the Wasserstein
distance) that
one can reasonably conjecture to be optimal, thus significantly improving
previous findings in the literature.
As a necessary step in our analysis, we
also derive new lower bounds for variances of Poisson functionals.
\end{abstract}


\smallskip



\tableofcontents

\section{Introduction}\label{intro}

\subsection{Overview and motivation}

We consider a Poisson process (Poisson random measure) $\eta$ on a measurable
space $(\BX,\cX)$, with $\sigma$-finite intensity measure $\lambda$,
see \cite{Kallenberg}.
Let $F=f(\eta)$ be a measurable function of $\eta$.
The {\em Poincar\'e inequality} (see \cite{LaPe11, Wu00} and the references therein)
states that the variance of a  square-integrable Poisson functional $F$ can be bounded as
\begin{align}\label{eq:Poincare}
\BV F\le \BE \int (D_xF)^2\,\lambda(\dint x),
\end{align}
where the {\em difference operator} $D_xF$ is defined as $D_xF:=f(\eta+\delta_x)-f(\eta)$.
Here, $\eta+\delta_x$ is the configuration arising by adding to $\eta$
a point at $x\in\BX$. Consequently,
a small expectation of $\|DF\|^2$ leads to small fluctuations of $F$,
where $DF$ is a short-hand notation for the mapping (discrete gradient)
$x\mapsto D_xF$ (also depending on $\eta$) and where
$\|v\|:=(\int v^2\,\dint\lambda)^{1/2}$ denotes the $L^2(\lambda)$-norm
of a measurable function $v:\BX\rightarrow\R$. The principal aim of this paper
is
to combine estimates of the type \eqref{eq:Poincare} with a suitable version
of {\it Stein's method}
(see e.g.\ \cite{NP11}), in order to establish explicit bounds on the normal
approximation of a general functional of the type $F = f(\eta)$. To do so, we
will partially
follow the route pioneered in references \cite{Chatt09, NoPeRe09},
in the framework of the normal approximation of functionals of Gaussian fields.

Indeed, the estimate \eqref{eq:Poincare} can be regarded as the Poisson
space counterpart of the famous {\em Chernoff-Nash-Poincar\'e inequality} of
Gaussian analysis (see \cite{Chernoff, Nash}), stating that, if $X =
(X_1,...,X_d)$ is an i.i.d.\ standard Gaussian vector
and $f$ is a smooth mapping on $\R^d$, then
\begin{align}\label{eq:PoincareG}
\BV f(X) \le \BE \|\nabla f(X)\|^2.
\end{align}
Motivated by problems in random matrix theory, Chatterjee \cite{Chatt09} has
extended \eqref{eq:PoincareG} to a {\em second order Poincar\'e inequality},
by proving the following bound:  if $f$ is twice differentiable, then
(for a suitable constant $C$ uniquely depending  on the variance of $f(X)$)
\begin{equation}\label{e:sopg}
d_{TV}(f(X), N)\leq C\,  \BE\left[\|{\rm Hess}\, f(X)\|_{op}^{4}\right]^{1/4} \times \BE\left[\|{\nabla f}(X)\|^{4}\right]^{1/4},
\end{equation}
where $d_{TV}$ is the total variation distance between the laws of two random
variables, $\|\cdot \|_{op}$
stands for the usual operator norm, and $N$ is a Gaussian random variable with
the same mean
and variance as $f(X)$. The main intuition behind relation \eqref{e:sopg} is
the following:
if the $L^4$ norm of $\|{\rm Hess}\, f(X)\|_{op}$ is negligible with respect to
that of $\|\nabla f(X)\|$,
then $f$ is close to an affine transformation, and therefore the distribution
of $f(X)$ must be
close to Gaussian. This class of second order results has been further
generalized in \cite{NoPeRe09}
to the framework of functionals $F$ of an infinite-dimensional isonormal
Gaussian process $X$ over a
Hilbert space $\mathfrak{H}$,
in which case the estimate \eqref{e:sopg} becomes (with obvious notation)
\begin{equation}\label{e:sopgg}
d_{TV}(F, N)\leq C\,  \BE\left[\|D^2 F\|_{op}^{4}\right]^{1/4}
\times \BE\left[\|DF\|_{\mathfrak{H}}^{4}\right]^{1/4},
\end{equation}
where $D$ and $D^2$ stand, respectively, for the first and second
Malliavin derivatives associated with $X$ (see also \cite[Chapter
5]{NP11}). We notice immediately
that, in general, the estimates \eqref{e:sopg}--\eqref{e:sopgg} yield
suboptimal rates of convergence,
that is: if $F_n$ is a sequence of centred smooth functionals of $X$ such
that $\BV F_n =: v^2_n \to \infty$ and $\tilde{F}_n = F_n/ v_n$ converges to
$N$ in distribution, then \eqref{e:sopg}--\eqref{e:sopgg} often
yield an upper bound on the quantity $d_{TV}(\tilde{F}_n, N)$ of the order of
$v_n^{-1/2}$, and not (as expected)
of $v_n^{-1}$; see for instance the examples discussed in \cite[Section 6]{NoPeRe09}.

In this paper we shall establish and apply a new class of second order
Poincar\'e inequalities,
involving general square-integrable functionals of the Poisson process $\eta$.
The counterparts of the operators $\nabla$ and
${\rm Hess}$ appearing in \eqref{e:sopg} will be, respectively, the difference
operator $D$,
and the {\em second order difference operator} $D^2$, which acts on a random variable
$F$ by generating the symmetric random mapping
$D^2F:\BX\times\BX\rightarrow\R, (x,y) \mapsto D^2_{x,y}F:=D_y(D_xF)$.
In view of the discrete nature of $D$ and $D^2$, our bounds will have a significantly
more complex structure than those appearing in
\eqref{e:sopg}--\eqref{e:sopgg}.
We will see that our estimates yield the presumably optimal rates of
convergence
(that is, rates of convergence proportional to the inverse of the square root
of the variance) in the normal approximation of non-linear functionals of
Poisson shot-noise processes, as well as in
two geometric applications displaying a {\em stabilizing}
nature (see e.g.\ \cite{PenroseYukich2005, SchreiberSur}). All
these rates have previously been outside the scope of 
existing techniques.

Our approach relies heavily on the normal approximation results proved in
\cite{EichelsbacherThaele2013, PSTU10, Schulte2012}, which are in turn derived
from a combination of Stein's method and Malliavin calculus. However,
in order to apply these results as efficiently as possible, we need
to establish a general {\em Mehler's formula} for Poisson processes (see
\cite{Priv09} for a special case),
providing a representation of the inverse Ornstein-Uhlenbeck generator in terms of
the {\em thinned} Poisson process.
The  development and application of this formula
is arguably one of the crucial contributions of our work.

\subsection{Main results}

Our main findings will provide upper bounds on the {\em Wasserstein distance} and the
{\em Kolmogorov distance} between the law of
a standardized Poisson functional and a that of a standard normal random variable.
Here, the Wasserstein distance between the laws of two random variables $Y_1,Y_2$
is defined as
$$
d_W(Y_1,Y_2)=\sup_{h\in {\operatorname{Lip}}(1)} |\BE h(Y_1)-\BE h(Y_2)|,
$$
where ${\operatorname{Lip}}(1)$ is the set of all functions
$h:\R\to\R$ with a Lipschitz-constant less than or equal to one.
The Kolmogorov distance between the laws of $Y_1,Y_2$ is given by
$$
d_K(Y_1,Y_2)=\sup_{x \in \R} |\BP(Y_1\leq x)-\BP(Y_2\leq x)|.
$$
This is the supremum distance between the distribution functions of $Y_1$ and $Y_2$.

Our bound on the Wasserstein distance $d_W(F,N)$ (where $F$ is a Poisson
functional with zero mean
and unit variance, and $N$ is a standard Gaussian random variable) is stated
in the
forthcoming Theorem \ref{thm:maindW}, and is expressed in terms of the
following three parameters,
whose definition involves exclusively the random functions $DF$ and $D^2F$:
\begin{align*}
\gamma_1 & :=  4\bigg[\int \big[\BE (D_{x_1}F)^2 (D_{x_2}F)^2\big]^{1/2}
\big[\BE (D_{x_1,x_3}^2F)^2(D_{x_2,x_3}^2F)^2\big]^{1/2}\, \lambda^3(\dint(x_1,x_2,x_3))
\bigg]^{1/2},\\
\gamma_2 & := \bigg[\int \BE (D_{x_1,x_3}^2F)^2(D_{x_2,x_3}^2F)^2 \,\lambda^3(\dint(x_1,x_2,x_3))
\bigg]^{1/2},\\
\gamma_3 & := \int \BE |D_xF|^3 \, \lambda(\dint x).
\end{align*}

As is customary, throughout the paper we shall use the following notation:
$L^2_\eta$ indicates
the class of all square-integrable functionals of the Poisson measure $\eta$;
by $\dom D$
we denote the collection of those $F\in L^2_\eta$ such that
\begin{align}\label{2.31}
\BE  \int (D_xF)^2\, \lambda(\dint x)<\infty.
\end{align}

\begin{theorem}\label{thm:maindW}
Let $F\in \dom D$ be such that $\BE F=0$ and $\BV F=1$, and let $N$
be a standard Gaussian random variable. Then,
$$
d_W(F,N) \leq \gamma_1+\gamma_2+\gamma_3.
$$
\end{theorem}

The numbers $\gamma_1$ and $\gamma_2$ control the size of
the fluctuations of the second order difference operator $D^2F$
in a relative and an absolute way. Therefore, a small value of $\gamma_1+\gamma_2$
indicates that $F$ is close to an element of the {\em first Wiener chaos}
of $\eta$, that is,
of the $L^2$ space generated by the linear functionals of
$\hat{\eta}:=\eta-\lambda$ (see e.g.\ \cite{LaPe11, PeTa}).
Moreover, a small value of $\gamma_3$ heuristically indicates that the projection of $F$
on the first Wiener
chaos of $\eta$ is close in distribution to a Gaussian random variable
(see e.g.\ \cite[Corollary 3.4]{PSTU10}).

In order to state our bound on the Kolmogorov distance $d_K(F,N)$,
we will need the following additional terms (carrying heuristic
interpretations similar
to those of $\gamma_1, \gamma_2, \gamma_3$):
\begin{align*}
\gamma_4 & := \frac{1}{2} \big[\BE F^4\big]^{1/4} \int \big[\BE (D_xF)^4\big]^{3/4} \, \lambda(\dint x), \allowdisplaybreaks\\
\gamma_5 & := \bigg[\int \BE(D_xF)^4 \, \lambda(\dint x)\bigg]^{1/2}, \allowdisplaybreaks\\
\gamma_6 & := \bigg[\int 6\big[\BE(D_{x_1}F)^4\big]^{1/2} \big[\BE(D^2_{x_1,x_2}F)^4\big]^{1/2}
+3\BE(D^2_{x_1,x_2}F)^4 \, \lambda^{2}(\dint(x_1,x_2))\bigg]^{1/2}.
\end{align*}

\begin{theorem}\label{thm:mainKolmogorov}
Let $F\in \dom D$ be such that $\BE F=0$ and $\BV F=1$, and let $N$ be a standard
Gaussian random variable. Then,
$$
d_K(F,N) \leq \gamma_1+\gamma_2+\gamma_3+\gamma_4+\gamma_5+\gamma_6.
$$
\end{theorem}

In view of \cite{PeccatiZheng2010}, our approach
can be extended so as to yield bounds
for the normal approximation of multivariate Poisson functionals,
while \cite{Peccati2011} allows in principle to deal with
Poisson approximations in the total variation distance.
Details and applications of such extensions will be reported elsewhere.

A first application of our general bounds concerns the asymptotic analysis of
Poisson functionals enjoying some weak form of {\em stabilization}
(see \cite{PenroseYukich2001,SchreiberSur}, as well as Section \ref{ss:literature}(c) below).
Our main result in this respect is Theorem \ref{thm:Stabilizing}, showing how the
`second order interactions' associated with a given Poisson functional
can be quantified in order to yield explicit bounds in normal approximations, without making
any further assumptions on the state space.
In order to motivate the reader, we
shall now present an important consequence of our Theorem \ref{thm:Stabilizing}.
For $t\geq 1$, we let $\eta_t$ be a Poisson process with intensity measure
$\lambda_t=t\lambda$, with $\lambda$ a fixed finite measure  on $\BX$.


\begin{proposition} \label{introex}
Let $F_t\in L^2_{\eta_t}$, $t\geq 1$, and assume there are finite constants
$p_1,p_2,c>0$ such that
\begin{equation}\label{e:i1}
\BE |D_xF_t|^{4+p_1}\leq c, \quad \lambda\text{-a.e. }x\in \BX, \quad t\geq 1,
\end{equation}
and
\begin{equation}\label{e:i2}
\BE |D^2_{x_1,x_2}F_t|^{4+p_2}\leq c,
\quad \lambda^2\text{-a.e. }(x_1,x_2)\in \BX^2, \quad t\geq 1.
\end{equation}
Moreover, assume that $\BV F_t/t >v$, $t\geq 1$, with $v>0$ and that
\begin{equation}\label{e:i3}
m:=\sup_{x\in\BX,\ t\geq 1} \int \BP(D_{x,y}^2F_t\neq 0)^{p_2/(16+4p_2)}\,\lambda_t(\dint y)<\infty.
\end{equation}
Let $N$ be a standard Gaussian random variable. Then, there exists a finite
constant $C$,
depending uniquely on $c,p_1,p_2,v, m$ and $\lambda(\BX)$, such that
$$
\max\left\{ d_W\bigg(\frac{F_t-\BE F_t}{\sqrt{\BV F_t}},N\bigg) ,
\,\,d_K\bigg(\frac{F_t-\BE F_t}{\sqrt{\BV F_t}},N\bigg)\right\} \leq C\, t^{-1/2},
\quad t\geq 1.
$$
\end{proposition}

The crucial assumption \eqref{e:i3} is intimately connected with
the theory of {\em stabilization}, developed in
\cite{BaryshnikovYukich2005,Penrose07,PenroseYukich2001,PenroseYukich2002,PenroseYukich2005}
and many other references; see \cite{SchreiberSur} for a survey.
Indeed, if $\BX\subset \R^d$ is compact and $\lambda$ equals the restriction of the Lebesgue measure to $\BX$,
then \eqref{e:i3} requires to bound
\begin{align*}
\int_{\BX} t\,\BP(D^2_{x,y}F_t\ne 0)^\alpha\,\dint y,
\end{align*}
uniformly in $x\in\BX$ and $t\ge 1$ for suitable $\alpha>0$, a task which is often simplified by some sort of
translation invariance of $D^2F_t$. Assume, for instance, that there exist finite random variables
$R_t(x,\eta_t)$ ({\em radii of stabilization}), such that
$D_xF_t$ does only depend on the restriction of $\eta_t$ to the ball
$B^d(x,R_t(x,\eta_t))$, where $B^d(x,r)$ is our generic notation for a ball with radius $r$
centred at $x$: formally, this means that, for every $t\geq 1$ and every $x\in \BX$,
\[
D_xf_t(\eta_t) = D_xf_t(\eta_t\cap B^d(x,R_t(x,\eta_t)) ),
\]
where $F_t\equiv f_t(\eta_t)$. Then, we need to show that
\begin{align*}
\sup_{x\in\BX,\ t\geq 1}\int_{\BX} t\,\BP\big(y\in B^d(x,R_t(x,\eta_t)) \mbox{\ or \  }  R_t(x,\eta_t+\delta_y)\neq R_t(x,\eta_t)\big)^\alpha\, \dint y<\infty.
\end{align*}
This is a close relative of the concept of {\em strong stabilization}
(or {\em add-one cost stabilization}) introduced in \cite{PenroseYukich2001}.

In Section \ref{sec:StochasticGeometry} we shall illustrate the power of
Proposition \ref{introex} and its more general version, Theorem \ref{thm:Stabilizing},
by fully developing two geometric applications, namely optimal
Berry-Esseen bounds in the normal approximation of:

\begin{itemize}
\item[(i)] statistics (including the total edge count and the total length) based on a $k$-nearest
neighbour graph with a Poisson input (generalising and improving previous
estimates from
\cite{AvramBertsimas1993,BaryshnikovYukich2005,BickelBreiman1983,Penrose07,PenroseYukich2001,PenroseYukich2005});

\item[(ii)] the intrinsic
  volumes of $k$-faces
associated with a Poisson-Voronoi tessellation (improving
the results in \cite{AvramBertsimas1993, BaryshnikovYukich2005, Heinrich1994, Penrose07, PenroseYukich2001, PenroseYukich2005}).

\end{itemize}

In our opinion, the new connection between the Stein-Malliavin approach and the
theory of stabilization
has a great potential for further generalisations and applications.

A second application (of a completely different nature) will be developed in
Section \ref{sec:ShotNoise},
where we apply our main theorems to study
integrals of non-linear functionals $\varphi(X_t)$ of general
Poisson {\em shot-noise processes} $(X_t)_{t\in\BY}$. Here, $\BY$ is
a measurable space and $X_t$ is a first order Wiener-It\^o integral
w.r.t.\ $\hat\eta$, whose integrand deterministically depends on the parameter $t\in\BY$.
Such Wiener-It\^o integrals provide fundamental examples of random fields, see
\cite{SpodBook} for a recent survey.
The numbers $\gamma_i$ can  be controlled by the
first and second derivative of $\varphi$ and our bounds yield
optimal rates of convergence in central limit theorems.
We will now illustrate our results with the (very) special example of
a {\em Ornstein-Uhlenbeck L\'evy process} (see e.g.\ \cite{BNS,PSTU10}).
To do so, we let $\eta$ be a Poisson process on $\R \times \R$
with intensity measure $\lambda(\dint u, \dint x)=\nu(\dint u) \, \dint x$,
where
the (L\'evy) measure $\nu$ on $\R$ satisfies
$$
\int |u|^\alpha \, \nu(\dint u)<\infty, \quad \alpha\in\{1,4+4p\},
$$
for some $p>0$. Define
$$
X_t:=\int \I\{x\le t\}ue^{-(t-x)}\,\hat{\eta}(\dint (u,x)),\quad t\in\R,
$$
and
$$
F_T:=\int^T_0 \varphi(X_t)\,\dint t,\quad T>0,
$$
where $\varphi:\R\to\R$ is a smooth function satisfying, for some $C'>0$,
$$
|\varphi(r)|+|\varphi'(r)|+|\varphi''(r)| \le C' \big(1+|r|^p\big), \quad r\in\R.
$$

\begin{proposition} \label{introexOU}
Let $(X_t)_{t\in\R}$ and $F_T$ be as above and let
$N$ be a standard Gaussian random variable. Assume that
\begin{equation}\label{eq:VariancePoissonShotNoise2}
\BV F_T \geq \sigma T, \quad T\geq t_0,
\end{equation}
for some $\sigma,t_0>0$. Then there is a constant $C>0$
depending on $\sigma$, $p$, $C'$ and $t_0$ such that
$$
d_K\bigg(\frac{F_T-\BE F_T}{\sqrt{\BV F_T}}, N\bigg) \leq \frac{C}{\sqrt{T}},
\quad T\geq t_0.
$$
\end{proposition}

A sufficient condition for \eqref{eq:VariancePoissonShotNoise2} to hold is that
$\tilde\varphi(r):=\BE[\varphi(X_0+r)-\varphi(X_0)]$
satisfies  $\int^u_0 \tilde\varphi(r)\,\dint r\ne 0$
for some $u>0$ in the support of $\nu$, see Lemma \ref{l8.4} and the
subsequent discussion.
Proposition \ref{introexOU} largely extends the CLTs proved in
\cite{EichelsbacherThaele2013, PSTU10, PeccatiZheng2010}, that only considered linear and quadratic functionals.


For the sake of conciseness, in Section 7 and Section 8 we will only present estimates involving the Kolmogorov distance $d_K$; since the bound in Theorem 1.1 is actually simpler than that in Theorem 1.2, analogous estimates hold also for the Wasserstein distance.

\subsection{Further connections with the existing literature}\label{ss:literature}

\noindent(a) {\it Malliavin-Stein techniques}. The present paper represents
the latest instalment in a very active direction of research,
based on the combined use of Stein's method and Malliavin calculus
in order to deduce explicit bounds for probabilistic approximations on the
Poisson space.  We refer the reader to \cite{PSTU10, PeccatiZheng2010}
for the first bounds of this type in the context of CLTs in the
Wasserstein and smoother distances, and to \cite{ReitznerSchulte} for the
first panoply of results indicating that Malliavin-Stein techniques
can be of great value for deriving quantitative CLTs in a geometric setting.
Other relevant contributions in this area are: \cite{EichelsbacherThaele2013, Schulte2012}
for explicit bounds in the Kolmogorov distance; \cite{HLS13} for applications
to additive functionals of Boolean models;
\cite{BP, DFR, LacPec_a, LacPec_b, LasPenSchTha}
for applications to geometric $U$-statistics and other non-linear functionals;
\cite{SchulteAAM} for CLTs in the framework of Poisson-Voronoi approximations;
\cite{Peccati2011, SchTh2012} for Poisson approximations and related geometric
applications. It is important to notice that all these references are based on
bounds involving not only the discrete gradient $D$, but also the so-called
{\em pseudo-inverse of the Ornstein-Uhlenbeck generator}, denoted in what
follows by $L^{-1}$  (see \eqref{eq:L-1} for a definition). Dealing with $L^{-1}$
is usually a very delicate issue, and typically requires excellent
estimates on the kernels appearing in a given Wiener-It\^o chaotic decomposition
(Wiener chaoses are indeed the eigenspaces of $L^{-1}$; see again \eqref{eq:L-1}).
While quantitative CLTs with optimal rates can still be obtained even for random
variables with an infinite chaotic expansion (see e.g.\ \cite{HLS13}),
the involved computation can be of a daunting technicality.
The estimates appearing in Theorem 1.1 and Theorem 1.2 above are the first
bounds obtained by Malliavin-Stein techniques that have a purely geometric nature,
and therefore do not require 
detailed knowledge of the chaotic decomposition of the underlying random variable $F$.

\medskip

\noindent(b) {\it Other second order inequalities}. An early attempt at
developing second-order Poincar\'e inequalities on the Poisson space
can be found in V\' iquez \cite{Viquez12}. One should note
that the results developed in this work are quite limited in scope,
as the author does not make use of Mehler's formula, in such a way
that the only treatable examples are those referring to
random variables living in a fixed Wiener chaos.

\medskip

\noindent(c) {\it Stabilization}. As made clear by the title
and by the previous discussion, we regard the hypothesis
\eqref{e:i3} of Proposition \ref{introex}
as a weak form of stabilization. The powerful and
far-reaching concept of stabilization in the context of central limit theorems was introduced in its actual form
by Penrose and Yukich in \cite{PenroseYukich2001, PenroseYukich2002} and
Baryshnikov and Yukich in \cite{BaryshnikovYukich2005},
building on the set of techniques introduced by Kesten and Lee \cite{kestenlee}
and Lee \cite{lee}.
This notion typically applies to a collection of
geometric functionals $\{F_t : t\geq 1\}$ of the type
\begin{equation}\label{e:f1}
F_t = \int_H h_t (x, \eta_t\cap H) \, \eta_t(\dint x), \quad t\geq 1,
\end{equation}
where $H\subset \R^d$ has finite Lebesgue measure, $\eta_t$ is a Poisson
process on $\R^d$ with
intensity $t\ell_d$ (where $\ell_d$ denotes Lebesgue measure),
$h_1 : H\times {\bf N} \to \R$ is some measurable
translation-invariant mapping (with {\bf N} indicating the class of $\sigma$-finite
configurations on $\R^d$; see Section 2),
and $h_t(x,\mu):=h_1(t^{1/d}x,t^{1/d}\mu)$. There are many versions
of stabilization.
Add-one cost stabilization (see \cite{PenroseYukich2001,PenroseYukich2002} and
the discussion after Proposition \ref{introex}) requires
the differences
$D_x F_t = F_t(\eta_t+\delta_x)-F_t$, to stabilize around around any $x\in H$
in a similar sense as described after Proposition \ref{introex}. {\em Exponential stabilization} requires stabilization of the functions
$h_t (x, \eta_t\cap H)$ as well as an exponential tail-behaviour
of the associated radii of stabilization; see \cite{PenroseYukich2005} for
more details.

One important result in the area that is relevant for our
paper (see e.g.\ \cite[Section 2]{PenroseRosoman2008}, \cite[Theorem 2.1]{PenroseYukich2005}
or \cite[Theorem 4.26]{SchreiberSur}) is that, if $\BV F_t \sim t$, some moment conditions are satisfied, and $F_t$
are exponentially stabilizing, then the random variables
$\tilde{F}_t:=(F-\BE F_t)/\sqrt{\BV F_t}$
verify a CLT, when $t\to\infty$, with a rate of convergence in the
Kolmogorov distance of the order of $t^{-1/2} A(t)$, where $A(t)$ is some
positive function slowly diverging to infinity
(e.g., $A(t) =(\log t)^\alpha$, for $\alpha>0$);
to the best of our knowledge, the question of whether the
factor $A(t)$ could be
removed at all has remained open until now.
If the functionals $F_t$ appearing
in  Proposition \ref{introex} are given in the form \eqref{e:f1}, then it is not difficult
(albeit quite technical) to translate \eqref{e:i1}--\eqref{e:i3} in terms of
assumptions on the kernels $h_t$, see also Remark \ref{Remark62}. We do
not give the details here.

However, as demonstrated by Proposition \ref{introex}
and its generalizations, our approach enjoys at least two fundamental
advantages with respect to classical stabilization techniques: (i) the
assumptions in our quantitative results do not require that the
functionals $F_t$ are represented in the special form \eqref{e:f1},
and (ii) the rates of convergence given by our estimates seem to be
systematically better (when both approaches apply).
As already discussed, in Section \ref{ss:kn} we will provide
optimal Berry-Esseen bounds for the total edge length of the $k$-nearest
neighbour graph over Poisson inputs; this is a typical stabilizing
functional for which our rates  improve those in the existing literature.
The advantage of our method seems to be that we exploit Stein's method
using operators from stochastic analysis that are intrinsically associated
with the model at hand, and therefore
we do not need to rely on a discretised version of the problem.

\medskip

\noindent(d) {\it Iteration of Efron-Stein}. The idea of
proving normal approximation results by controlling second order interactions
between random points is also successfully applied in \cite{Chatt08},
where the author combines Stein's method with an iteration of the Efron-Stein
inequality, in order to deduce explicit rates of convergence in a number of CLTs
involving geometric functionals over binomial inputs. At this stage,
it is difficult to compare our techniques with those of \cite{Chatt08}
for mainly two reasons: (i) while it is possible to apply the results
from \cite{Chatt08} to compute Berry-Ess\'een bounds for Poisson functionals
(as demonstrated in the recent reference  \cite{ChaSen2013}, dealing with the
total length of the minimal spanning tree over a Poisson input),
this operation can only be accomplished via a discretisation argument
that is completely absent from our approach, and (ii) the techniques of
\cite{Chatt08} are specifically devised for deducing rates of convergence in
the Wasserstein distance, whereas bounds in the Kolmogorov distance
are only obtained by using the standard relation between $d_K$ and $d_W$
(see e.g.\ \cite[formula (C.2.6)]{NP11}), a strategy which cannot be expected, in general, to yield the optimal rates.

\subsection{Plan}

This paper is organized as follows. After some preliminaries in
Section \ref{sec:preliminaries}, Mehler's formula is established in
Section \ref{sec:mehler}. It is the crucial argument in the proofs of
Theorem \ref{thm:maindW} and \ref{thm:mainKolmogorov}, which are given in
Section \ref{sec:proofs}. Section \ref{sec:Variancepositive} contains the
proof of several lower bounds for variances, that are helpful in applications,
in order to ensure that variances do not degenerate. In Section \ref{s:stab}
we consider the normal approximation of stabilizing Poisson functionals.
Section \ref{sec:StochasticGeometry} and Section \ref{sec:ShotNoise} contain
the applications of our results to problems from stochastic geometry and
to non-linear functionals of shot-noise processes.

\section{Preliminaries}\label{sec:preliminaries}

The reader is referred to the monograph \cite{Priv09} as well as to the paper \cite{LaPe11} for any unexplained definition or result.

Let $(\BX,\cX)$ be an arbitrary measurable space and let $\lambda$ be a
$\sigma$-finite measure on $\BX$ such that $\lambda(\BX)>0$. For $p>0$ and
$n\in\N$ we denote by $L^p(\lambda^n)$ the set of all measurable functions
$f: \BX^n\to\R$ such that $\int |f|^p \, \dint\lambda^n<\infty$.
We call a function $f:\BX^n\to\R$ {\it symmetric} if it is invariant under
permutations of its arguments, and denote by $L^p_s(\lambda^n)$ the set of
all $f\in L^p(\lambda^n)$ such that there exists a symmetric function $\hat{f}$ verifying $f = \hat{f}$, $\lambda^n$-a.e. In particular, the class $L^2_s(\lambda^n)$ is a Hilbert subspace of $L^2(\lambda^n)$. For $f,g\in L^2(\lambda^n)$
we define the inner product $\langle f,g\rangle_n:=\int f g \, \dint\lambda^n$
and the norm $\|f\|_n:=(\int f^2 \, \dint\lambda^n)^{1/2}$.


For the rest of this paper, let $(\Omega, \mathcal{F},\BP)$ be the underlying
probability space. For $p>0$, let $L^p(\BP)$ denote the space of all
random variables $Y: \Omega \rightarrow \R$ such that $\BE|Y|^p<\infty$.
Let $\eta$ be a Poisson process in $\BX$ with intensity measure $\lambda$
defined on the underlying probability space $(\Omega, \mathcal{F},\BP)$.
As usual we interpret $\eta$ as a random element in the space
$\bN$ of integer-valued $\sigma$-finite measures $\mu$
on $\BX$ equipped with the smallest $\sigma$-field $\mathcal{N}$
making the mappings $\mu\mapsto\mu(B)$ measurable for all
$B\in\mathcal{X}$; see \cite{Kallenberg} or \cite{LaPe11}. We write $\hat{\eta}$ for the compensated random (signed) measure $\eta-\lambda$. By $L^p_\eta$, $p>0$, we denote the space of all random variables $F\in L^p(\BP)$
such that $F=f(\eta)$, $\BP$-a.s., for some measurable function
$f: \bN\rightarrow\R$; such a function $f$ (which is uniquely determined by $F$ up to sets of $\BP$-measure zero) is customarily called a {\em representative} of $F$.


Let $f:\bN\rightarrow\R$ be a  measurable function. For $n\in\N$ and
$x_1,\dots,x_n\in\BX$ we define
\begin{align}\label{Dsymmetric}
D^n_{x_1,\ldots,x_n}f(\mu):=\sum_{J \subset \{1,2,\ldots,n\}}(-1)^{n-|J|}
f \Big(\mu+\sum_{j\in J}\delta_{x_j}\Big), \quad \mu\in \bN,
\end{align}
where $|J|$ denotes the number of elements of $J$. This
shows that the $n$-th order {\em difference operator} $D^n_{x_1,\ldots,x_n}$
is symmetric in $x_1,\ldots,x_n$, and that
$(\mu,x_1,\ldots,x_n)\mapsto D^n_{x_1,\ldots,x_n}f(\mu)$
is measurable. For fixed $\mu\in\bN$ the latter mapping is abbreviated as $D^nf(\mu)$.
For $F\in L^2_\eta$ with representative $f$ we define $D^nF=D^nf(\eta)$.
By the multivariate Mecke equation (see e.g.\ \cite[formula (2.10)]{LaPe11}) this definition
does $\lambda^n$-a.e.\ and $\BP$-a.s.\ not depend on the
choice of the representative $f$.

For $F,G\in L^2_\eta$ the
{\em Fock space representation} derived in \cite{LaPe11} says that
\begin{align}\label{covcha}
\BE [F G]=\BE[F]\BE[G]+\sum^\infty_{n=1}n! \langle f_n,g_n \rangle_n,
\end{align}
where $f_n:=\frac{1}{n!}\BE D^nF$, $g_n:=\frac{1}{n!}\BE D^nG$ and $f_n,g_n\in L^2_s(\lambda^n)$.
If $F=G$, \eqref{covcha} yields a formula for the variance of $F$.

Let us denote by $I_n(g)$ the $n$-th order {\em Wiener-It\^o integral} of $g\in L^2_s(\lambda^n)$ with respect to $\hat{\eta}$. Note that the Wiener-It\^o integrals satisfy the isometry and orthogonality relation
\begin{equation}\label{eq:IsometryIntegrals}
\BE I_n(g)I_m(h)= n! \, \I\{m=n\} \langle f,g\rangle_n, \quad g\in L^2_s(\lambda^n), h\in L^2_s(\lambda^m), n,m\in\N.
\end{equation}
 It is a well-known fact that every $F\in L^2_\eta$ admits a representation of the type
\begin{equation}\label{WIC}
F=\BE F + \sum_{n=1}^\infty I_n(f_n),
\end{equation}
where $f_n =\frac{1}{n!}\BE D^nF$ and the right-hand side
converges in $L^2(\BP)$. Such a representation is known as {\em Wiener-It\^o chaos expansion} of $F$;
in this general and explicit form the result was proved in \cite{LaPe11}.

Given $F$ as in \eqref{WIC}, we write $F\in\dom D$ if
\begin{align}\label{2.2}
\sum_{n=1}^\infty n n!\|f_n\|^2_n<\infty.
\end{align}
In this case, we have that, $\BP$-a.s.\ and for $\lambda$-a.e.\ $x\in \BX$,
\begin{align}\label{2.3}
D_xF=\sum_{n=1}^\infty n I_{n-1}(f_n(x,\cdot)),
\end{align}
cf.\ \cite{LaPe11} and the references given there.
Applying the Fock space relation \eqref{covcha} to $D_xF$, it is easy to see
that \eqref{2.2} is actually equivalent to the integrability condition \eqref{2.31}; see \cite[Lemma 3.1]{PeTh13}. The restriction of $D$ to the space $\dom D$ is usually called the {\it Malliavin derivative operator} associated with the Poisson measure $\eta$. 

One can think of the difference operator $D$ as an operator mapping a
random variable to a random function. The
{\em Skorohod integral} (or {\em Kabanov-Skorohod integral} -- see \cite{Kab}) $\delta$ maps a random function $g$ from the space $L^2_\eta(\BP\otimes\lambda)$ of all elements of $L^2(\BP\otimes\lambda)$ that are $\BP\otimes\lambda$-a.e.\ of the form $g(\omega,x)=\tilde{g}(\eta(\omega),x)$ for some measurable $\tilde{g}$ to a random variable.
Its domain $\dom\delta$ is the set of all $g\in L^2_\eta(\BP\otimes\lambda)$ having Wiener-It\^o chaos expansions
\begin{align}\label{2.4}
g(x)=\sum_{n=0}^\infty I_n(g_n(x,\cdot)), \quad x\in\BX,
\end{align}
with $I_0(c)=c$ for $c\in\R$ and measurable functions $g_n: \BX^{n+1}\to\R$, $n\in\N\cup\{0\}$, that are symmetric
in the last $n$ variables and satisfy
\begin{equation}\label{eq:domdelta}
\sum_{n=0}^\infty (n+1)! \|\tilde{g}_n\|_{n+1}^2<\infty.
\end{equation}
Here $\tilde{h}: \BX^m\to\R$ stands for the symmetrization
$$
\tilde{h}(x_1,\hdots,x_m)=\frac{1}{m!}
\sum_{\sigma\in{\operatorname{Per}}(m)} h(x_{\sigma(1)}, \hdots,x_{\sigma(m)})
$$
of $h:\BX^m\to\R$, where ${\operatorname{Per}}(m)$ denotes the group
of all permutations of $\{1,\hdots,m\}$. Now the Skorohod integral of $g\in\dom\delta$
is defined as
$$
\delta(g)=\sum_{n=0}^\infty I_{n+1}(\tilde{g}_n).
$$


The third operator we will use in the following is the so-called
{\em Ornstein-Uhlenbeck generator}. Its domain is given by
all $F\in L^2_\eta$ satisfying
$$
\sum^\infty_{n=1}n^2 n! \|f\|^2_n<\infty.
$$
In this case one defines
$$
LF=- \sum_{n=1}^\infty n I_n(f_n).
$$
The (pseudo) {\em inverse} $L^{-1}$ of $L$ is given by
\begin{equation}\label{eq:L-1}
L^{-1}F:=- \sum_{n=1}^\infty \frac{1}{n} I_n(f_n).
\end{equation}
Note that the random variable $L^{-1}F$ is a well-defined element of $L^2_\eta$ for every $F\in L^2_\eta$.


The following integration by parts formula
\cite{LaPe11,Priv09} shows that the
difference operator and the
Skorohod integral can be seen as dual operators.

\begin{lemma}\label{lem:integrationbyparts}
Let $F\in\dom D$ and $g\in\dom\delta$. Then
$$
\BE\int D_xF \, g(x) \, \lambda(\dint x)=\BE[F \, \delta(g)].
$$
\end{lemma}

In the previous formula, one needs that $F\in \dom D$ (although $D_xF$ is still defined otherwise). Obviously, $\I\{F>t\}$ is in $L^2_\eta$ for any $F\in L^2_\eta$ and $t\in\R$, but it is unclear whether $\I\{F>t\}\in\dom D$ whenever the underlying intensity $\lambda$ is such that $\lambda(\BX) = \infty$. To overcome this difficulty, we shall need the following special integration by parts formula, for which we assume slightly more than $g\in\dom D$ (because of the missing symmetrization in \eqref{eq:assumptiong}).

\begin{lemma}\label{lem:integrationbypartsRefined}
Let $F\in L^2_\eta$, $s\in \R$ and $g\in L^2_\eta(\BP\otimes \lambda)$ such that
\begin{equation}\label{eq:assumptiong}
\sum_{n=0}^\infty (n+1)! \|g_n\|_{n+1}^2<\infty
\end{equation}
and assume that $D_x\I\{F>s\} g(x)\geq 0$, $\BP$-a.s., $\lambda$-a.e.\ $x\in\BX$. Then $g\in\dom \delta$ and
$$
\BE\int D_x\I\{F>s \} \, g(x) \, \lambda(\dint x)=\BE[\I\{F>s\} \, \delta(g)].
$$
\end{lemma}
\begin{proof}
By the Cauchy-Schwarz inequality, \eqref{eq:assumptiong} implies \eqref{eq:domdelta} so that $g\in\dom\delta$. In Lemma 2.3 in \cite{Schulte2012} the assertion is stated for $g\in\dom D$ with finite chaos expansions. But the proof still holds for $g\in L^2(\BP\otimes \lambda)$ satisfying \eqref{eq:assumptiong}, since an application of the Cauchy-Schwarz inequality shows that one can interchange summation and integration in equation (9) in \cite{Schulte2012}.
\end{proof}

We next state a basic isometry property of the Skorohod integral.
Although special cases of this result are well-known (see e.g.\ \cite{Priv09})
we give the proof for the sake of completeness.

\begin{proposition}\label{prop:isometry}
Let $g\in L^2_\eta(\BP\otimes \lambda)$ be such that
\begin{align}\label{2.56}
\BE \iint (D_yg(x))^2\,\lambda(\dint x)\,\lambda(\dint y)<\infty.
\end{align}
Then, $g$ satisfies \eqref{eq:assumptiong} and $g\in\dom\delta$, and
\begin{align}\label{eq:iso}
\BE \delta(g)^2= \BE\int g(x)^2\,\lambda(\dint x)
+\BE \iint D_yg(x)D_xg(y)\,\lambda(\dint x)\,\lambda(\dint y).
\end{align}
\end{proposition}

\begin{proof}
Suppose that $g$ is given by \eqref{2.4}. Assumption \eqref{2.56} implies that
$g(x)\in\dom D$ for $\lambda$-a.e.\ $x\in\BX$. We therefore deduce
from \eqref{2.3} that
\begin{align*}
h(x,y):=D_y g(x)=\sum^\infty_{n=1} n I_{n-1}(g_n(x,y,\cdot))
\end{align*}
$\BP$-a.s.\ and for $\lambda^2$-a.e.\ $(x,y)\in\BX^2$.
Using assumption \eqref{2.56} together with \eqref{eq:IsometryIntegrals}, one infers that
\begin{equation}\label{eq:boundSkorohod}
\sum^\infty_{n=1} nn! \|\tilde {g}_n\|^2_{n+1}\leq \sum^\infty_{n=1} n n! \|g_n\|^2_{n+1} = \BE \iint (D_yg(x))^2\,\lambda(\dint x)\,\lambda(\dint y) <\infty,
\end{equation}
yielding that $g$ satisfies \eqref{eq:assumptiong} and that $g\in {\rm dom} \, \delta$. Now define
\begin{align*}
g^{(m)}(x):=\sum^m_{n=0} I_n(g_n(x,\cdot)),\quad x\in\BX,
\end{align*}
for $m\in\N$. Then
\begin{align*}
\BE \delta(g^{(m)})^2=\sum^m_{n=0}\BE I_{n+1}(\tilde{g}_n)^2
=\sum^m_{n=0}(n+1)! \|\tilde{g}_n\|^2_{n+1}.
\end{align*}
Using the symmetry of the functions $g_n$ it is easy
to see that the latter sum equals
\begin{align}\label{2.11}
\sum^m_{n=0} n!\int g^2_n\,\dint\lambda^{n+1}
+\sum^m_{n=1} n n!\iint g_n(x,y,z)
g_n(y,x,z)\,\lambda^{2}(\dint (x,y))\,\lambda^{n-1}(\dint z),
\end{align}
with the obvious interpretation of the integration with respect to
$\lambda^0$. On the other hand, we have from \eqref{2.3} that
$$
D_y g^{(m)}(x)=\sum^m_{n=1} n I_{n-1}(g_n(x,y,\cdot)),
$$
so that
\begin{align*}
\BE\int g^{(m)}(x)^2\,\lambda(\dint x)
+\BE \iint D_yg^{(m)}(x)D_xg^{(m)}(y)\,\lambda(\dint x)\,\lambda(\dint y)
\end{align*}
coincides with \eqref{2.11}. Hence
\begin{align}\label{2.9}
\BE \delta(g^{(m)})^2= \BE\int g^{(m)}(x)^2\,\lambda(\dint x)
+\BE \iint D_yg^{(m)}(x)D_xg^{(m)}(y)\,\lambda(\dint x)\,\lambda(\dint y).
\end{align}
From the equality part in \eqref{eq:boundSkorohod} it follows that $h_m(x,y):=D_y g^{(m)}(x)$ converges to $h$ in
$L^2(\BP\otimes\lambda^2)$. Similarly,
$h'_m(x,y):=D_xg^{(m)}(y)$ converges towards $h'(x,y):=D_x g(y)$.
Together with
$$
\BE \int g^{(m)}(x)^2\,\lambda(\dint x)=\sum^m_{n=1} n! \|g_n\|^2_{n+1} \to \sum^\infty_{n=1} n! \|g_n\|^2_{n+1} = \BE \int g(x)^2\,\lambda(\dint x)
$$
as $m\to\infty$ we can now conclude that the right-hand side of
\eqref{2.9} tends to the right-hand side of the asserted
identity \eqref{eq:iso}.
On the other hand,
\begin{align*}
\BE \delta(g-g^{(m)})^2=\sum^\infty_{n=m+1}\BE I_{n+1}(\tilde{g}_n)^2
=\sum^\infty_{n=m+1}(n+1)! \|\tilde{g}_n\|^2_{n+1}
\le \sum^\infty_{n=m+1}(n+1)! \|g_n\|^2_{n+1}\to 0
\end{align*}
as $m\to\infty$. This concludes the proof.
\end{proof}

We will also exploit the following consequence of Proposition \ref{prop:isometry}.

\begin{corollary}\label{cor:SkorohodIsometry}
Let $g\in\dom\delta$. Then
\begin{align}\label{2.53}
\BE \delta(g)^2\le \BE\int g(x)^2\,\lambda(\dint x)
+\BE \iint (D_yg(x))^2\,\lambda(\dint x)\,\lambda(\dint y).
\end{align}
\end{corollary}

\begin{proof}
We can assume that the right-hand side
of \eqref {2.53} is finite. Now the Cauchy-Schwarz inequality
implies that
$$
\BE \iint D_y g(x)D_x g(y)\,\lambda(\dint x)\,\lambda(\dint y)
\le \BE \iint (D_yg(x))^2\,\lambda(\dint x)\,\lambda(\dint y),
$$
and the result follows from Proposition \ref{prop:isometry}.
\end{proof}

In the sequel, we use the following Poincar\'e inequality to establish that a Poisson functional is in $L^2_\eta$ and to bound its variance.

\begin{proposition}\label{prop:Poincare}
For $F\in L^1_\eta$,
\begin{equation}\label{eq:PoincareL1}
\BE F^2 \leq (\BE F)^2 + \BE \int (D_xF)^2 \, \lambda(\dint x).
\end{equation}
In particular, $F\in L^2_\eta$ if the right-hand side is finite.
\end{proposition}

\begin{proof}
For $F\in L^2_\eta$, \eqref{eq:PoincareL1} is a special case of Theorem 1.2 in \cite{LaPe11}. We extend it to $F\in L^1_\eta$ by the following truncation argument. For $s>0$ we define
$$
F_s=\I\{F>s\}s+ \I\{-s\leq F\leq s\}F -\I\{F<-s\}s 
$$
By definition of $F_s$ we have $F_s\in L^2_\eta$ and $|D_xF_s|\leq |D_xF|$ for $\lambda$-a.e. $x\in\BX$. Together with the Poincar\'e inequality for $L^2$-functionals we obtain that
$$
\BE F_s^2 \leq (\BE F_s)^2+\BE\int (D_xF_s)^2 \, \lambda(\dint x)\leq  (\BE F_s)^2+\BE\int (D_xF)^2 \, \lambda(\dint x).
$$
By the monotone convergence theorem and the dominated convergence theorem, respectively, we have that $\BE F_s^2\to\BE F^2$ and $\BE F_s\to\BE F$ as $s\to\infty$. Hence letting $s\to\infty$ in the previous inequality yields the assertion.
\end{proof}

\section{Mehler's formula}\label{sec:mehler}

For $F\in L^1_\eta$ with representative $f$ we define
\begin{align}\label{ePt}
P_s F:=\int \BE[f(\eta^{(s)}+\mu)\mid\eta] \, \Pi_{(1-s)\lambda}(\dint\mu),\quad s\in[0,1],
\end{align}
where $\eta^{(s)}$ is a $s$-thinning of $\eta$
and where
$\Pi_{\lambda'}$ denotes the distribution of a Poisson process with intensity
measure $\lambda'$. The thinning  $\eta^{(s)}$ can be defined by first representing
$\eta$ as a finite or countable sum of Dirac measures at random
points (possible by the construction of a Poisson process) and then
removing these points independently of
each other with probability $1-s$, see \cite[p.\ 226]{Kallenberg}. This does also show that the mapping $(\omega,s)\mapsto P_sF(\omega)$
can be assumed to be measurable. Since
\begin{align}\label{convww}
\Pi_\lambda=\BE\int \I\{\eta^{(s)}+\mu\in\cdot\} \, \Pi_{(1-s)\lambda}(\dint\mu),
\end{align}
the definition of $P_sF$ does almost surely not depend
on the choice of the representative of $F$.
Equation \eqref{convww} implies in particular that
\begin{align}\label{eq32}
\BE P_sF =\BE F,\quad F\in L^1_\eta,
\end{align}
while Jensen's inequality and \eqref{convww} yield the contractivity property
\begin{align}\label{contr}
  \BE[|P_sF|^p]\le \BE[|F|^p],\quad  s\in[0,1],\, F\in L^1_\eta,
\end{align}
for every $p\geq 1$.

\begin{lemma}\label{lemmaMehler0} Let $F\in L^2_\eta$.
Then, for all $n\in\N$ and $s\in[0,1]$,
\begin{align}\label{91}
D^n_{x_1,\dots,x_n}(P_s F)=s^nP_s D^n_{x_1,\dots,x_n}F, \quad
\lambda^n\text{-a.e. } (x_1,\ldots,x_n)\in\BX^n,\;\BP\text{-a.s.}
\end{align}
In particular
\begin{align}\label{eq:EDnPs}
\BE D^n_{x_1,\dots,x_n} P_s F =s^n \BE D^n_{x_1,\dots,x_n} F, \quad
\lambda^n\text{-a.e. } (x_1,\dots,x_n)\in\BX^n.
\end{align}
\end{lemma}

\begin{proof}
Let $s\in[0,1]$. 
To begin with, we assume that the representative of $F$
is given by $f(\mu)=e^{-\int v \, \dint\mu}$ for some $v:\BX\rightarrow[0,\infty)$
such that $\lambda(\{v>0\})<\infty$.
By the definition of a thinning,
$$
\BE\big[e^{-\int v \, \dint \eta^{(s)}}\mid \eta\big]=
\exp\bigg[\int \log\big((1-s)+s e^{-v(y)}\big) \, \eta(\dint y)\bigg],
$$
and it follows from Lemma 12.2 in \cite{Kallenberg} that
$$
\int \exp\bigg(- \int v \, \dint\mu\bigg) \, \Pi_{(1-s)\lambda}(\dint\mu)
=\exp\bigg[-(1-s)\int (1-e^{-v}) \, \dint\lambda \bigg].
$$
Hence, the definition \eqref{ePt} of the operator $P_s$ implies that
the following function $f_s:\bN\to\R$ is a representative of $P_sF$:
\begin{align*}
f_s(\mu):=\exp\bigg[-(1-s)\int \big(1-e^{-v}\big)\,\dint\lambda\bigg]
\exp\bigg[\int \log\big((1-s)+s e^{-v(y)}\big)
\, \mu(\dint y)\bigg].
\end{align*}
Therefore we obtain for any $x\in\BX$ that
\begin{align*}
D_xP_sF=f_s(\eta+\delta_x)-f_s(\eta)=s\big(e^{-v(x)}-1\big)f_s(\eta)
=s\big(e^{-v(x)}-1\big)P_sF.
\end{align*}
This identity can be iterated to yield for all $n\in\N$ and all
$(x_1,\ldots,x_n)\in\BX^n$ that
\begin{align*}
D^n_{x_1,\ldots,x_n}P_sF=s^n\prod^n_{i=1}\big(e^{-v(x_i)}-1\big)P_sF.
\end{align*}
On the other hand we have $\BP$-a.s.\ that
\begin{align*}
  P_sD^n_{x_1,\ldots,x_n}F=P_s\prod^n_{i=1}\big(e^{-v(x_i)}-1\big)F
= \prod^n_{i=1}\big(e^{-v(x_i)}-1\big)P_sF,
\end{align*}
so that \eqref{91} 
holds for Poisson functionals of the given form.

By linearity, \eqref{91} 
extends to all $F$ with a representative in the set $\mathbf{G}$
of all linear combinations of functions $f$ as above.
In view of \cite[Lemma 2.1]{LaPe11}, for any $F\in L^2_\eta$ with representative $f$ there are $f^k\in \mathbf{G}$, $k\in\N$,
satisfying $F^k:=f^k(\eta)\to F=f(\eta)$ in $L^2(\BP)$ as $k\to\infty$. The contractivity property \eqref{contr} implies that
\begin{align*}
\BE (P_sF^k-P_sF)^2=\BE (P_s(F^k-F))^2\le \BE(F^k-F)^2\to 0,
\end{align*}
as $k\to\infty$. Taking $B\in\cX$ with $\lambda(B)<\infty$,
it therefore follows from \cite[Lemma 2.3]{LaPe11} that
\begin{align*}
\BE\int_{B^n} |D^n_{x_1,\ldots,x_n}P_sF^k-D^n_{x_1,\ldots,x_n} P_sF| \,
\lambda(\dint (x_1,\ldots,x_n))\to 0,
\end{align*}
as $k\to\infty$. On the other hand we obtain from the Fock space representation \eqref{covcha} that $\BE |D^n_{x_1,\ldots,x_n}F|<\infty$
for $\lambda^n$-a.e.\ $(x_1,\ldots,x_n)\in\BX^n$, so that
linearity of $P_s$ and \eqref{contr} imply
\begin{align*}
  \BE\int_{B^n} |P_sD^n_{x_1,\ldots,x_n}&F^k-P_sD^n_{x_1,\ldots,x_n}F|
\,\lambda(\dint (x_1,\ldots,x_n))\\
&\le\int_{B^n} \BE |D^n_{x_1,\ldots,x_n}(F^k-F)| \, \,\lambda(\dint (x_1,\ldots,x_n)).
\end{align*}
Again, this latter integral tends to $0$ as $k\to\infty$.
Since \eqref{91} 
holds for any $F^k$ we obtain that
\eqref{91} 
holds $\BP\otimes(\lambda_B)^n$-a.e., and hence
also $\BP\otimes\lambda^n$-a.e.

Taking the expectation in \eqref{91} and using \eqref{eq32} proves \eqref{eq:EDnPs}.
\end{proof}

The following result yields a pathwise representation of the
inverse Ornstein-Uhlenbeck operator using the operators
$\{P_s:s\in[0,1]\}$.

\begin{theorem}\label{MehlerOU} Let $F\in L^2_\eta$. If $\BE F=0$,
then we have $\BP$-a.s.\ that
\begin{align}\label{invOU}
L^{-1}F=-\int^1_0 s^{-1}P_sF \, \dint s.
\end{align}
\end{theorem}

\begin{proof}
Assume that $F$ is given as in \eqref{WIC}. Because of \eqref{contr}, we have that $P_sF\in L^2_\eta$. Applying \eqref{WIC} to $P_sF$ and using \eqref{eq:EDnPs}, we therefore infer that
\begin{equation}\label{e55}
P_sF=\BE F+\sum_{n=1}^\infty s^n I_n(f_n) =\sum_{n=1}^\infty s^n I_n(f_n),  \quad\BP\text{-a.s.},\, s\in[0,1],
\end{equation}
where we have used the fact that $F$ is centred. Relation \eqref{e55} can be used to show that the integral on the right-hand side of \eqref{invOU} is $\BP$-a.s.\ finite and defines a square-integrable random variable. To see this, just apply (in order) Jensen's inequality, \eqref{e55} and \eqref{eq:IsometryIntegrals} to deduce that
\[
\BE\left[ \left( \int_0^1 s^{-1} |P_sF| \, \dint s \right)^2\right]\leq \BE\left[  \int_0^1 s^{-2} \BE |P_sF|^2\,  \dint s \right]
= \sum^\infty_{n=1}n!\|f_n\|^2_n\int^1_0s^{2n-2} \, \dint s<\infty.
\]
Now,
$$
L^{-1}\left(\sum^m_{n=1} I_n(f_n)\right)= -\sum^m_{n=1}\frac{1}{n} I_n(f_n)=-\int^1_0s^{-1}\sum^m_{n=1}s^n I_n(f_n) \, \dint s,\quad
m\ge 1.
$$
Since $L^{-1}$ is a continuous operator from $L_\eta^{2}$ into itself, and in view of \eqref{eq:L-1}, we need to show that the right-hand
side of the above expression converges in $L_\eta^2$, as $m\to\infty$, to the right-hand of side
of \eqref{invOU}. Taking into account \eqref{e55} we hence have to show that
\begin{align*}
R_m:=\int^1_0s^{-1}\bigg(P_sF-\sum^m_{n=1}s^n I_n(f_n)\bigg) \dint s
=\int^1_0s^{-1}\bigg(\sum^\infty_{n=m+1}s^n I_n(f_n)\bigg) \dint s
\end{align*}
converges in $L_\eta^2$ to zero. We obtain
\begin{align*}
  \BE R^2_m\le\int^1_0s^{-2}\BE\bigg(\sum^\infty_{n=m+1}s^n I_n(f_n)\bigg)^2 \, \dint s
  =\sum^\infty_{n=m+1}n!\|f_n\|^2_n\int^1_0s^{2n-2} \, \dint s,
\end{align*}
which tends to zero as $m\to\infty$.
\end{proof}

We also record the following important consequence of \eqref{invOU}.

\begin{corollary}
For every $F\in L^2_\eta$ such that $\BE F=0$,
\begin{align}\label{eq:DL-1}
-DL^{-1}F=\int^1_0 P_sDF \, \dint s,\quad \BP\otimes\lambda\text{-a.e.},
\end{align}
and
\begin{align}\label{eq:D2L-1}
-D^2L^{-1}F=\int^1_0 s \, P_sD^2F \, \dint s,\quad \BP\otimes\lambda^2\text{-a.e.}
\end{align}
\end{corollary}

\begin{proof} In the proof of Theorem \ref{MehlerOU} we have seen that
\begin{align*}
\BE \int^1_0 s^{-2} (P_sF)^2 \, \dint s
\le \sum^\infty_{n=1}n!\|f_n\|^2_n<\infty.
\end{align*}
In particular,
\begin{align}\label{10001}
\int^1_0 s^{-1} |P_sF| \, \dint s<\infty, \quad\BP\text{-a.s.}
\end{align}
Furthermore, Lemma \ref{lemmaMehler0} implies that,
for $\lambda$-a.e.\ $x\in \BX$,
\begin{align*}
\BE \int^1_0 s^{-1} |D_xP_sF| \, \dint s
=\int^1_0 \BE |P_sD_xF| \, \dint s\le \BE |D_xF|,
\end{align*}
where we have used $\BE |D_xF|<\infty$ for $\lambda$-a.e.\ $x\in\BX$ (which follows from \eqref{covcha}) and the contractivity property \eqref{contr}
to get the inequality. We obtain that
\begin{align}\label{10002}
\int^1_0 s^{-1} |D_xP_sF| \, \dint s<\infty,\quad \lambda\text{-a.e. } x\in\BX,\;\BP\text{-a.s.}
\end{align}
Denoting by $f_s$ a representative of $P_sF$, we derive from
\eqref{10001} and \eqref{10002} that
\begin{align*}
\int^1_0 s^{-1} |f_s(\eta+\delta_x)| \, \dint s<\infty \quad \text{and}\quad
\int^1_0 s^{-1} |f_s(\eta)| \, \dint s<\infty,
\end{align*}
for $\lambda$-a.e.\ $x\in\BX$ and $\BP$-a.s. Hence, the difference
operator of the right-hand side of \eqref{invOU} is the integrated
difference operator and \eqref{eq:DL-1} follows from
Lemma \ref{lemmaMehler0}.

The proof of \eqref{eq:D2L-1} is similar, using that
$\BE |D^2_{x_1,x_2}F|<\infty$ for $\lambda^2$-a.e.\ $(x_1,x_2)\in\BX^2$.
\end{proof}


\bigskip
In the following, we shall refer to the identity \eqref{invOU} as {\it Mehler's formula}. Note that this formula can be written as
$$
L^{-1}F=-\int^\infty_0 T_sF \, \dint s,
$$
where $T_sF:=P_{e^{-s}}F$ for $s\ge 0$. The family $\{T_s:s\ge 0\}$ of operators describes
a special example of {\em Glauber dynamics}. From \eqref{e55}, it follows in particular that
\begin{equation}\label{e:mehlertrue}
T_sF=\BE F+\sum_{n=1}^\infty e^{-ns} I_n(f_n), \quad\BP\text{-a.s.},\,s\ge 0,
\end{equation}
which was proven for the special case of a finite Poisson process with a diffuse intensity
measure in \cite{Priv09}. One should note that, for the sake of brevity, in this paper we slightly deviate from the standard terminology adopted in a Gaussian framework, where the equivalent of \eqref{e:mehlertrue} and \eqref{invOU} are called, respectively, `Mehler's formula' and `integrated Mehler's formula' (see e.g. \cite{NP11}).

We conclude this section with two useful inequalities.

\begin{lemma}\label{lem:boundDLinversep}
For $F\in L^2_\eta$ and $p\geq 1$ we have
$$
\BE |D_xL^{-1}F|^p \leq  \BE |D_xF|^p, \quad \lambda\text{-a.e. }x\in\BX,
$$
and
$$
\BE |D^2_{x,y}L^{-1}F|^p
\leq  \BE |D^2_{x,y}F|^p, \quad \lambda^2\text{-a.e. }(x,y)\in\BX^2.
$$
\end{lemma}
\begin{proof}
Let $f: \bN\to\R$ be a representative of $F$. Combining \eqref{eq:DL-1}
and the definition of $P_s$ leads to
$$
\BE |D_xL^{-1}F|^p = \BE \left|\int_0^1
\int \BE[D_xf(\eta^{(s)}+\mu) \mid \eta] \, \Pi_{(1-s)\lambda}(\dint\mu) \, \dint s\right|^p, \quad \lambda\text{-a.e. }x\in\BX.
$$
An application of Jensen's inequality with respect to the integrals
and the conditional expectation yields
$$
\BE |D_xL^{-1}F|^p \leq \BE \int_0^1 \int \BE[|D_xf(\eta^{(s)}+\mu)|^p \mid \eta] \,
\Pi_{(1-s)\lambda}(\dint\mu) \, \dint s, \quad \lambda\text{-a.e. }x\in\BX.
$$
Because of \eqref{convww} the right-hand side can be simplified to $\BE|D_xf(\eta)|^p$,
which concludes the proof of the first inequality. By \eqref{eq:D2L-1}
and analogous arguments as above, we obtain that, for $\lambda^2$-a.e. $(x,y)\in\BX^2$,
\begin{align*}
\BE |D^2_{x,y}L^{-1}F|^p & = \BE \left| \int_0^1 \int s \,
\BE[ D^2_{x,y}f(\eta^{(s)}+\mu) \mid \eta] \, \Pi_{(1-s)\lambda}(\dint\mu) \, \dint s \right|^p\\
& \leq \BE \int_0^1 \int  \BE[ |D^2_{x,y}f(\eta^{(s)}+\mu)|^p \mid \eta] \,
\Pi_{(1-s)\lambda}(\dint \mu) \, \dint s=\BE |D^2_{x,y}F|^p,
\end{align*}
which is the second inequality.
\end{proof}

\section{Proofs of Theorem \ref{thm:maindW}
and Theorem \ref{thm:mainKolmogorov}}\label{sec:proofs}

\subsection{Ancillary computations}

Our proofs of Theorem \ref{thm:maindW} and Theorem \ref{thm:mainKolmogorov} are based
on the following bounds, taken respectively from \cite[Theorem 3.1]{PSTU10} and \cite[Theorem 3.1]{EichelsbacherThaele2013},
concerning the Wasserstein and the Kolmogorov distance between the law of a given $F\in L^2_\eta$ satisfying $\BE F=0$ and $F\in\dom D$,
and the law of a standard Gaussian random variable $N$:

\begin{align}\label{eq:boundPSTU}
d_W(F ,N)
\le \BE \big|1-\int (D_xF) (-D_xL^{-1}F) \, \lambda(\dint x)\big|
+\BE \int (D_xF)^2|D_xL^{-1}F| \, \lambda(\dint x),
\end{align}

\begin{align}\label{eq:EichelsbacherThaele}
d_K(F,N) & \leq \BE\big| 1-\int (D_xF) (-D_xL^{-1}F) \, \lambda(\dint x)\big| +\frac{\sqrt{2\pi}}{8}\BE\int (D_xF)^2 |D_xL^{-1}F| \, \lambda(\dint x)\\
& \hskip -1cm +\frac{1}{2} \BE \int (D_xF)^2 |F| |D_xL^{-1}F| \, \lambda(\dint x)+\sup_{t\in\R} \BE\int (D_x\I\{F>t\}) (D_xF) |D_xL^{-1}F| \, \lambda(\dint x). \notag
\end{align}
One should note that the estimate \eqref{eq:EichelsbacherThaele} improves a previous result in \cite{Schulte2012}. The above bounds \eqref{eq:boundPSTU}--\eqref{eq:EichelsbacherThaele} are rather general.
However, both can be quite difficult to evaluate if one uses the representation \eqref{eq:L-1} of the inverse Ornstein-Uhlenbeck generator since this requires the explicit knowledge of the kernels $f_n$, $n\in\N$, of the Fock space representation, which is usually not the case for a given Poisson functional. Our main tool for overcoming this problem is Mehler's formula. In the following result it is combined with the Poincar\'e inequality, in order to control the first summand on the right-hand sides of \eqref{eq:boundPSTU} and \eqref{eq:EichelsbacherThaele}.

\begin{proposition}\label{prop:firstpart}
For $F,G\in \dom D$ with $\BE F = \BE G=0$, we have
\begin{align*}
& \BE \bigg(\BC(F,G)-\int (D_xF) (-D_xL^{-1}G) \, \lambda(\dint x)\bigg)^2\\
&  \leq  3 \int  \big[\BE (D_{x_1,x_3}^2F)^2(D_{x_2,x_3}^2F)^2\big]^{1/2} \big[\BE (D_{x_1}G)^2 (D_{x_2}G)^2\big]^{1/2}\, \lambda^3(\dint(x_1,x_2,x_3))\\
& \quad + \int  \big[\BE (D_{x_1}F)^2(D_{x_2}F)^2\big]^{1/2} \big[\BE (D_{x_1,x_3}^2G)^2 (D_{x_2,x_3}^2G)^2\big]^{1/2}\, \lambda^3(\dint(x_1,x_2,x_3))\\
& \quad + \int  \big[\BE (D_{x_1,x_3}^2F)^2(D_{x_2,x_3}^2F)^2\big]^{1/2} \big[\BE (D_{x_1,x_3}^2G)^2 (D_{x_2,x_3}^2G)^2\big]^{1/2}\, \lambda^3(\dint(x_1,x_2,x_3)).
\end{align*}
\end{proposition}

\begin{proof}
We can of course assume that the three integrals on the right-hand side of the inequality are finite -- otherwise there is nothing to prove. Let $f,g: \bN\to\R$ be representatives of $F$ and $G$. Combining \eqref{covcha} with \eqref{2.3} and \eqref{eq:L-1}, we have
$$\BE \int (D_xF) (-D_xL^{-1}G) \, \lambda (\dint x) =\BC(F,G).$$
Start by assuming that
\begin{equation}\label{eq:assumptionFinite}
\int \big|D_y\big((D_xF) (-D_xL^{-1}G)\big) \big| \, \lambda (\dint x)<\infty, \quad \lambda\text{-a.e. } y\in\BX, \quad \BP\text{-a.s.}
\end{equation}
Then, the integral
$$
\int (D_xf(\eta+\delta_y)) (-D_x\tilde{g}(\eta+\delta_y)) \, \lambda (\dint x),
$$
where $\tilde{g}$ is a representative of $L^{-1}G$, exists and is finite $\BP$-a.s.\ for $\lambda$-a.e.\ $y\in\BX$. Consequently, $\BP$-a.s.\ for $\lambda$-a.e.\ $y\in\BX$,
$$
D_y \int (D_xF) (-D_xL^{-1}G) = \int D_y (D_xF) (-D_xL^{-1}G)  \, \lambda (\dint x),
$$
and
\begin{align*}
\bigg|D_y  \int (D_xF) (-D_xL^{-1}G) \, \lambda (\dint x) \bigg|
\leq \int \big|D_y\big((D_xF) (-D_xL^{-1}G)\big) \big| \, \lambda (\dint x).
\end{align*}
Together with the Poincar\'e inequality (see Proposition \ref{prop:Poincare}) this yields
\begin{align*}
& A:= \BE\left(\BC(F,G)-\int (D_xF) (-D_xL^{-1}G) \, \lambda(\dint x)\right)^2\\
&\leq \BE\int \bigg(D_y  \int (D_xF) (-D_xL^{-1}G) \, \lambda (\dint x) \bigg)^2 \, \lambda(\dint y)\\
&\leq \BE\int \bigg(  \int \big|D_y\big((D_xF) (-D_xL^{-1}G)\big)\big| \, \lambda (\dint x) \bigg)^2 \, \lambda(\dint y):=B.
\end{align*}
Of course, if assumption \eqref{eq:assumptionFinite} is not satisfied, then the estimate $A\leq B$ (as defined above) continues (trivially) to hold. Since, for any $x,y\in\BX$,
\begin{align*}
& D_y((D_xF)(-D_xL^{-1}G))\\
& =(D^2_{x,y}F)(-D_xL^{-1}G)+(D_xF)(-D^2_{x,y}L^{-1}G)+(D^2_{x,y}F)(-D^2_{x,y}L^{-1}G),
\end{align*}
we obtain that
\begin{align}\label{e37}
\BE \bigg(\BC(F,G)-\int (D_xF) (-D_xL^{-1}G) \, \lambda(\dint x)\bigg)^2\le 3(I_1+I_2+I_3)
\end{align}
with
\begin{align*}
I_1&:=\BE\int \bigg(\int | (D^2_{x,y}F) (-D_{x}L^{-1}G)| \, \lambda(\dint x)\bigg)^2 \, \lambda(\dint y),\\
I_2&:=\BE\int \bigg(\int |(D_{x}F) (-D^2_{x,y}L^{-1}G)| \, \lambda(\dint x)\bigg)^2 \, \lambda(\dint y),\\
I_3&:=\BE\int \bigg(\int |(D^2_{x,y}F) (-D^2_{x,y}L^{-1}G)| \, \lambda(\dint x)\bigg)^2 \, \lambda(\dint y).
\end{align*}

\noindent We will now use Mehler's formula to derive upper bounds for $I_1$, $I_2$ and $I_3$. By combining \eqref{eq:DL-1} with the definition of $P_s$ and Fubini's Theorem, we see that
\begin{align*}
& \int | (D^2_{x,y}F) (-D_{x}L^{-1}G)| \, \lambda(\dint x)\\
 & = \int \left| D^2_{x,y}F \int_0^1 P_s D_xG \, \dint s \right| \, \lambda(\dint x) \allowdisplaybreaks\\
& = \int \left| D^2_{x,y}f(\eta) \int_0^1 \int \BE\left[D_x g(\eta^{(s)}+\mu)\bigm|  \eta\right] \,
\Pi_{(1-s) \, \lambda}(\dint\mu) \, \dint s \right| \, \lambda(\dint x)\\
& \leq  \int_0^1 \int \BE\bigg[\int |D^2_{x,y}f(\eta) D_x g(\eta^{(s)}+\mu)| \, \lambda(\dint x)  \biggm| \eta \bigg] \, \, \Pi_{(1-s)\lambda}(\dint\mu) \, \dint s.
\end{align*}
Now an application of Jensen's inequality with respect to the outer integrals and the conditional expectation as well as the Cauchy-Schwarz inequality lead to
\begin{align*}
& \bigg(\int |(D^2_{x,y}F) (-D_{x}L^{-1}G)| \, \lambda(\dint x)\bigg)^2\\
&  \leq \int_0^1 \int \BE\bigg[\int | D^2_{x_1,y}f(\eta) D^2_{x_2,y}f(\eta) D_{x_1}g(\eta^{(s)}+\mu) D_{x_2}g(\eta^{(s)}+\mu)| \, \lambda^2(\dint(x_1,x_2))  \biggm| \eta \bigg] \\
& \hskip 12cm
\Pi_{(1-s)\lambda}(\dint\mu) \, \dint s \\
&  = \int |D^2_{x_1,y}f(\eta) D^2_{x_2,y}f(\eta)|\\
& \quad\quad \int_0^1 \int \BE[|D_{x_1}g(\eta^{(s)}+\mu) D_{x_2}g(\eta^{(s)}+\mu)|  \bigm| \eta] \, \Pi_{(1-s)\lambda}(\dint\mu) \, \dint s \lambda^2(\dint(x_1,x_2))\\
&  \leq \int \left| D^2_{x_1,y}f(\eta) D^2_{x_2,y}f(\eta)\right| \\
& \hskip 1cm \bigg[\int_0^1\int \BE[(D_{x_1}g(\eta^{(s)}+\mu))^2 (D_{x_2}g(\eta^{(s)}+\mu))^2 \bigm| \eta] \, \, \Pi_{(1-s)\lambda}(\dint\mu) \, \dint s\bigg]^{1/2} \, \lambda^2(\dint(x_1,x_2)).
\end{align*}
Using the Cauchy-Schwarz inequality again, we obtain that
\begin{align*}
I_1 \leq \int & \bigg[\BE \int_0^1\int \BE[(D_{x_1}g(\eta^{(s)}+\mu))^2 (D_{x_2}g(\eta^{(s)}+\mu))^2 \bigm| \eta] \, \Pi_{(1-s)\lambda}(\dint\mu) \, \dint s\bigg]^{1/2}\\
& \big[\BE (D^2_{x_1,y}f(\eta))^2 (D^2_{x_2,y}f(\eta))^2\big]^{1/2} \, \lambda^3(\dint(x_1,x_2,y)).
\end{align*}
By \eqref{convww}, the first part of the integrand simplifies to $\big[\BE (D_{x_1}G)^2 (D_{x_2}G)^2\big]^{1/2}$ so that
$$
I_1 \leq \int  \big[\BE (D^2_{x_1,y}F)^2 (D^2_{x_2,y}F)^2\big]^{1/2} \,
\big[\BE (D_{x_1}G)^2 (D_{x_2}G)^2\big]^{1/2} \, \lambda^3(\dint(x_1,x_2,y)).
$$
For $I_2$ and $I_3$ we obtain in a similar way by using \eqref{eq:D2L-1} that
\begin{align*}
& \bigg(\int | (D_xF) (-D_{x,y}^2L^{-1} G)| \, \lambda(\dint x) \bigg)^2\\
& \leq\bigg(\int |D_xf(\eta)| \int_0^1 s \int \BE[|D_{x,y}^2g(\eta^{(s)}+\mu)| \bigm| \eta] \, \Pi_{(1-s)\lambda}(\dint\mu) \, \dint s \, \lambda(\dint x) \bigg)^2 \allowdisplaybreaks\\
& =\bigg( \int_0^1 \int \BE\bigg[ s \int | D_xf(\eta) D_{x,y}^2g(\eta^{(s)}+\mu)| \, \lambda(\dint x) \biggm| \eta \bigg] \, \Pi_{(1-s)\lambda}(\dint\mu) \, \dint s  \bigg)^2 \allowdisplaybreaks\\
& \leq \int_0^1 u^2 \, \dint u \times \int_0^1 \int \BE\bigg[ \int |D_{x_1}f(\eta) D_{x_2}f(\eta)D_{x_1,y}^2g(\eta^{(s)}+\mu)\\
& \hskip 4.5cm D_{x_2,y}^2g(\eta^{(s)}+\mu)|\, \lambda^2(\dint(x_1,x_2)) \biggm| \eta\bigg] \, \Pi_{(1-s)\lambda}(\dint\mu) \, \dint s \allowdisplaybreaks \\
& \leq \frac{1}{3}\int \bigg[\int_0^1 \int \BE[ (D_{x_1,y}^2g(\eta^{(s)}+\mu))^2 (D_{x_2,y}^2g(\eta^{(s)}+\mu))^2 \bigm| \eta] \, \Pi_{(1-s)\lambda}(\dint\mu) \, \dint s\bigg]^{1/2} \\
& \hskip 1.3cm \left| D_{x_1}f(\eta) D_{x_2}f(\eta)\right| \,  \lambda^2(\dint(x_1,x_2))
\end{align*}
and
\begin{align*}
& \bigg(\int | (D_{x,y}^2F) (-D_{x,y}^2L^{-1}G) | \, \lambda(\dint x) \bigg)^2\\
& \leq \bigg(\int | D_{x,y}^2f(\eta)| \int_0^1 \int \BE[s |D_{x,y}^2g(\eta^{(s)}+\mu)| \bigm| \eta] \, \Pi_{(1-s)\lambda}(\dint\mu) \, \dint s \, \lambda(\dint x) \bigg)^2 \allowdisplaybreaks\\
& = \bigg(\int_0^1 \int \BE\bigg[s \int | D_{x,y}^2f(\eta)  D_{x,y}^2g(\eta^{(s)}+\mu)|  \, \lambda(\dint x) \biggm| \eta\bigg] \, \Pi_{(1-s)\lambda}(\dint\mu) \, \dint s \bigg)^2\\
& \leq  \int_0^1 u^2 \, \dint u\times \int_0^1 \int \BE\bigg[ \int | D_{x_1,y}^2f(\eta) D_{x_2,y}^2f(\eta)  D_{x_1,y}^2g(\eta^{(s)}+\mu)\\
& \hskip 4.4cm D_{x_2,y}^2g(\eta^{(s)}+\mu)|  \, \lambda^2(\dint(x_1,x_2)) \biggm| \eta \bigg] \, \Pi_{(1-s)\lambda}(\dint\mu) \, \dint s \allowdisplaybreaks\\
& = \frac{1}{3}  \int \left| D_{x_1,y}^2f(\eta) D_{x_2,y}^2f(\eta)\right|\\
& \hskip 1cm  \int_0^1 \int \BE[|D_{x_1,y}^2g(\eta^{(s)}+\mu) D_{x_2,y}^2g(\eta^{(s)}+\mu)| \bigm| \eta] \, \Pi_{(1-s)\lambda}(\dint\mu) \, \dint s    \, \lambda^2(\dint (x_1,x_2))\allowdisplaybreaks\\
& \leq \frac{1}{3}  \int\left| D_{x_1,y}^2f(\eta) D_{x_2,y}^2f(\eta)\right| \\
& \hskip 1cm \bigg[ \int_0^1 \int \BE[(D_{x_1,y}^2g(\eta^{(s)}+\mu))^2 (D_{x_2,y}^2g(\eta^{(s)}+\mu))^2 \bigm| \eta] \, \Pi_{(1-s)\lambda}(\dint\mu) \, \dint s \bigg]^{1/2}   \, \lambda^2(\dint (x_1,x_2)).
\end{align*}
As before a combination of the Cauchy-Schwarz inequality and \eqref{convww} leads to
$$I_2 \leq \frac{1}{3}\int \big[\BE (D_{x_1}F)^2 (D_{x_2}F)^2\big]^{1/2} \big[\BE (D_{x_1,y}^2G)^2 (D_{x_2,y}^2G)^2\big]^{1/2} \, \lambda^3(\dint(x_1,x_2,y))$$
and
$$I_3 \leq \frac{1}{3}\int \big[\BE (D_{x_1,y}^2F)^2 (D_{x_2,y}^2F)^2\big]^{1/2} \big[\BE (D_{x_1,y}^2G)^2 (D_{x_2,y}^2G)^2\big]^{1/2} \, \lambda^3(\dint(x_1,x_2,y)).$$
Combining the inequalities for $I_1$, $I_2$ and $I_3$ with \eqref{e37} yields the assertion.
\end{proof}

\subsection{Proofs}

We can now proceed to the proof of our main results.

\begin{proof}[Proof of Theorem \ref{thm:maindW}]
For the first summand on the right-hand side of \eqref{eq:boundPSTU} we obtain, by the Cauchy-Schwarz inequality and Proposition \ref{prop:firstpart} in the case $G=F$, that
$$
\BE\big|1-\int (D_xF) (-D_xL^{-1}F) \, \lambda(\dint x) \big| \leq \sqrt{\BE \bigg(1 - \int (D_xF) (-D_xL^{-1}F) \, \lambda(\dint x) \bigg)^2}\leq \gamma_1+\gamma_2.
$$
For the second part of the bound in \eqref{eq:boundPSTU}, H\"older's inequality and Lemma \ref{lem:boundDLinversep} yield
\begin{align*}
\BE\int (D_xF)^2 |D_xL^{-1}F| \, \lambda(\dint x) & \leq \int \big[\BE |D_xF|^3\big]^{2/3} \big[\BE|D_xL^{-1}F|^3\big]^{1/3} \, \lambda(\dint x)\\
 & \leq \int \BE |D_xF|^3 \, \lambda(\dint x),
\end{align*}
which concludes the proof.
\end{proof}

\begin{proof}[Proof of Theorem \ref{thm:mainKolmogorov}]

Observe that the first and the second summand in \eqref{eq:EichelsbacherThaele} can be treated exactly as in the proof of Theorem \ref{thm:maindW}. H\"older's inequality and Lemma \ref{lem:boundDLinversep} yield that
\begin{align*}
\BE \int (D_xF)^2 |F| |D_xL^{-1}F| \, \lambda(\dint x) & \leq \int \big[\BE (D_xF)^4\big]^{1/2} \big[\BE F^4\big]^{1/4} \big[\BE|D_xL^{-1}F|^4\big]^{1/4} \, \lambda(\dint x)\\
& \leq \big[\BE F^4\big]^{1/4} \int \big[\BE (D_xF)^4\big]^{3/4} \, \lambda(\dint x) = 2\gamma_4.
\end{align*}
To conclude the proof, assume that $\gamma_5+\gamma_6 <\infty$ (otherwise, there is nothing to prove). We shall first show that the random function $g(x) := D_xF|D_xL^{-1}F|$ verifies the integrability condition $$A:= \BE\int g(x)^2\,\lambda(\dint x)
+\BE \iint (D_yg(x))^2\,\lambda(\dint x)\,\lambda(\dint y)<\infty.$$
By the trivial inequality $\left||a|-|b|\right|\leq |a-b|$, which is true for all $a,b\in\R$, we have that
$$
|D_y|D_xL^{-1}F||\leq |D^2_{x,y}L^{-1}F|.
$$
Thus, we infer that
$$
|D_y((D_xF) |D_xL^{-1}F|)|\leq |D^2_{x,y}F| \, |D_x L^{-1}F|+ |D_{x}F| \, |D^2_{x,y} L^{-1}F|
+ |D^2_{x,y}F| \, |D^2_{x,y} L^{-1}F|,
$$
whence
\begin{align*}
A & \leq \BE \int (D_xF)^2 (D_xL^{-1}F)^2 \, \lambda(\dint x)+ 3 \BE \iint (D^2_{x,y}F)^2 (D_x L^{-1}F)^2+ (D_{x}F)^2 (D^2_{x,y} L^{-1}F)^2\\
& \hskip 7.5cm + (D^2_{x,y}F)^2 (D^2_{x,y}L^{-1}F)^2 \, \lambda(\dint x) \, \lambda(\dint y).
\end{align*}
Now the Cauchy-Schwarz inequality and Lemma \ref{lem:boundDLinversep} yield that
\begin{align*}
A & \leq \int \BE (D_xF)^4 \, \lambda(\dint x)
   + 3 \iint \big[\BE (D^2_{x,y}F)^4\big]^{1/2} \big[\BE (D_x F)^4\big]^{1/2}\\ & \hskip 5.2cm + \big[\BE (D_{x}F)^4\big]^{1/2} \big[\BE (D^2_{x,y}F)^4\big]^{1/2} + \BE (D^2_{x,y}F)^4
\, \lambda(\dint x) \, \lambda(\dint y)\\
&\leq  \gamma_5^2+\gamma_6^2<\infty.
\end{align*}
By virtue of Proposition \ref{prop:isometry}, this yields that $g$ satisfies \eqref{eq:assumptiong}. The integration by parts formula in Lemma \ref{lem:integrationbypartsRefined} together with the Jensen inequality (and the fact that indicators are bounded by 1) now imply that
\begin{align*}
\BE \int D_x\I\{F>t\} (D_xF) |D_xL^{-1}F| \, \lambda(\dint x) & = \BE \I\{F>t\} \, \delta((DF) |DL^{-1}F|)\\
& \leq \big[\BE \delta((DF) |DL^{-1}F|)^2\big]^{1/2}.
\end{align*}
Finally, Corollary \ref{cor:SkorohodIsometry} and the upper bound for $A$ above imply that
$$
\BE \delta((DF) |DL^{-1}F|)^2 \leq A \leq \gamma_5^2+\gamma_6^2.
$$
Combining all these bounds with \eqref{eq:EichelsbacherThaele} concludes the proof.
\end{proof}

\begin{example}\label{ex:fc}{\rm Consider the simple case where $F= I_1(f)$ is an element of the first Wiener chaos of $\eta$, where the deterministic kernel $f\in L^2(\lambda)$ is such that $\|f\|_1=1$. Then, one has that $DF = f$, $D^2F =0$,
\[
\BE F^2=\|f\|_1^2=1 \quad\mbox{and}\quad \BE F^4 = 3 + \|f\|^4_{L^4(\lambda)},
\]
where we have implicitly applied a multiplication formula between Wiener-It\^o integrals such as the one stated in \cite[Theorem 6.5.1]{PeTa} and use the notation $\|f\|_{L^p(\lambda)}:=(\int |f|^p \, \dint\lambda)^{1/p}$, $p>0$. In this framework, Theorem \ref{thm:maindW} and Theorem \ref{thm:mainKolmogorov} imply, respectively, that
\[
d_W(F,N) \leq \|f\|^3_{L^3(\lambda)} \quad \mbox{and}\quad d_{K}(F,N) \leq \|f\|^3_{L^3(\lambda)}\times \bigg(1 +\frac{3^{1/4}}{2}+ \frac{\|f\|_{L^4(\lambda)}}{2}\bigg)+\|f\|^2_{L^4(\lambda)}.
\]
Analogous bounds for the Wasserstein distance can also be inferred from \cite[Corollary 3.4]{PSTU10}, whereas the bounds for the Kolmogorov distance can alternatively be deduced from the main results of \cite{Schulte2012}.

The most straightforward application of these bounds for the normal approximation of first order Wiener-It\^o integrals corresponds to the case where $\BX = \R_+$, $\lambda$ equals the Lebesgue measure, and $F = F_t = I_1(f_t)$, with $t>0$ and $f_t(x) = t^{-1/2} {\bf 1} \{x\leq t\}$. In this case, one has that $F_t $ is a rescaled centred Poisson random variable with parameter $t$, and the previous estimates become
\[
d_W(F_t,N) \leq \frac{1}{\sqrt{t}} \quad \mbox{and}\quad d_{K}(F_t,N) \leq \frac{1}{\sqrt{t}} \times \left(2+\frac{3^{1/4}}{2}+ \frac{1}{2t^{1/4}} \right),
\]
yielding rates of convergence (as $t\to \infty$) that are consistent with the usual Berry-Esseen estimates.
}
\end{example}

We conclude the section by recording a useful inequality, which we shall apply throughout the paper in order to bound the fourth moment of a given random variable $F$ in terms of the difference operator $DF$. In particular, such an estimate is crucial for dealing with the quantity $\gamma_4$ appearing in the statement of Theorem \ref{thm:mainKolmogorov}.

\begin{lemma}\label{lem:BoundFourthMoment}
Let $F\in L^2_\eta$ be such that $\BE F=0$ and $\BV F=1$. Then,
$$\BE F^4 \leq \max\bigg\{256 \bigg[ \int \big[\BE(D_zF)^4\big]^{1/2} \, \lambda(\dint z)\bigg]^2, 4\int \BE(D_zF)^4 \, \lambda(\dint z) +2 \bigg\}.$$
Moreover, if the right-hand side is finite, $F\in\dom D$.
\end{lemma}

\begin{proof}
By virtue of the Poincar\'e inequality (see Proposition \ref{prop:Poincare}) (as applied to the random variable $F^2$) and of a straightforward computation, we obtain
\begin{align*}
\BE F^4 & = \BV F^2 + (\BE F^2)^2 = \BV F^2 +1\\
& \leq \int \BE \big(D_z(F^2)\big)^2 \, \lambda(\dint z) +1 = \int \BE \big(2F(D_zF) +(D_zF)^2\big)^2 \, \lambda(\dint z) +1 \allowdisplaybreaks\\
& \leq \int 8 \BE F^2 (D_zF)^2 + 2\BE(D_zF)^4 \, \lambda(\dint z) +1\\
& \leq 8 \big[\BE F^4\big]^{1/2} \int \big[\BE(D_zF)^4\big]^{1/2} \, \lambda(\dint z) + 2\int \BE(D_zF)^4 \, \lambda(\dint z) +1\\
& \leq \max\bigg\{16 \big[\BE F^4\big]^{1/2} \int \big[\BE(D_zF)^4\big]^{1/2} \, \lambda(\dint z),4\int \BE(D_zF)^4 \, \lambda(\dint z) +2\bigg\},
\end{align*}
which implies the inequality. The conclusion $F\in\dom D$ follows from \eqref{2.31} and the Cauchy-Schwarz inequality.
\end{proof}

\section{Lower bounds for variances}\label{sec:Variancepositive}

\subsection{General bounds}

In what follows, we will apply our main bounds to sequences of standardized random variables of the form $(F- \BE F)/\sqrt{{\rm Var}\,(F)}$, where $F$ is e.g.\ some relevant geometric quantity. Our aim in this section is to prove several new analytical criteria, allowing one to deduce explicit lower bounds for variances. Our approach, which we believe is of independent interest, is based on the use of difference operators, and is perfectly tailored to deal with geometric applications.

We start by proving two criteria, ensuring that the variance of a given random variable is non-zero or greater than a constant, respectively.

\begin{lemma}
Let $F\in L^2_\eta$ and let $f:\bN\to\R$ be a representative of $F$. Then $\BV F=0$ if and only if
\begin{equation}\label{eq:conditionnondegenerate}
\BE[f(\eta+\sum_{i\in I_1}\delta_{x_i})-f(\eta+\sum_{i\in I_2}\delta_{x_i})]=0,
\quad \lambda^k\text{-a.e. }(x_1,\hdots,x_k)\in\BX^k,
\end{equation}
for all $k\in\N$ and $I_1,I_2\subset\{1,\hdots,k\}$ such that $I_1\cup I_2=\{1,\hdots,k\}$.
\end{lemma}

\begin{proof}
Let us assume that $\BV F=0$. Then it follows from \eqref{covcha} that
\begin{equation}\label{eq:assumptionzero}
\BE D^n_{x_1,\hdots,x_n}F=0, \quad \lambda^n\text{-a.e. }(x_1,\hdots,x_n)\in\BX^n,
\end{equation}
for all $n\in\N$. Now it can be shown by induction that
\begin{equation}\label{eq:consequencezero}
\BE[f(\eta+\sum_{i=1}^n \delta_{x_i})-f(\eta)]=0,
\quad \lambda^n\text{-a.e. }(x_1,\hdots,x_n)\in\BX^n.
\end{equation}
The case $n=1$ coincides with \eqref{eq:assumptionzero}. For $n\geq 2$ we have (using \eqref{Dsymmetric})
\begin{align*}
\BE D^n_{x_1,\hdots,x_n} F & =\BE \sum_{I\subset\{1,\hdots,n\}} (-1)^{n-|I|} f(\eta+\sum_{i\in I}\delta_{x_i})\\
& = \sum_{I\subset\{1,\hdots,n\}} (-1)^{n-|I|} \BE[f(\eta+\sum_{i\in I}\delta_{x_i})-f(\eta)]+\sum_{I\subset\{1,\hdots,n\}} (-1)^{n-|I|} \BE f(\eta).
\end{align*}
Here, the second summand is zero due to the alternating sign. The induction hypothesis yields that in the first sum only the summand for $I=\{1,\hdots,n\}$ remains for $\lambda^n$-a.e. $(x_1,\hdots,x_n)\in \BX^n$, which proves \eqref{eq:consequencezero}.

By \eqref{eq:consequencezero} we obtain that
$$\BE[f(\eta+\sum_{i\in I_1} \delta_{x_i})-f(\eta+\sum_{i\in I_2} \delta_{x_i})]=\BE[f(\eta+\sum_{i\in I_1} \delta_{x_i})-f(\eta)]-\BE[f(\eta+\sum_{i\in I_2} \delta_{x_i})-f(\eta)]=0$$
for $\lambda^n$-a.e.\ $(x_1,\hdots,x_n)\in\BX^n$.

The other direction holds since \eqref{eq:conditionnondegenerate} for all $k\in\N$ and all subsets $I_1,I_2$  implies \eqref{eq:assumptionzero} for all $n\in\N$, which is equivalent to $\BV F=0$.
\end{proof}

The next theorem provides a quantitative bound for the case that the variance is not zero.

\begin{theorem}\label{thm:Variancepositive}
Let $F\in L^2_\eta$ with a representative $f:\bN\to\R$ and assume that there are $k\in\N$, $I_1,I_2\subset \{1,\hdots,k\}$ with $I_1\cup I_2=\{1,\hdots,k\}$, $U\subset\BX^k$ measurable and $c>0$ such that
\begin{equation}\label{eq:lowerbound}
|\BE[f(\eta+\sum_{i\in I_1}\delta_{x_i})-f(\eta+\sum_{i\in I_2}\delta_{x_i})]|\geq c, \quad \lambda^k\text{-a.e. }(x_1,\hdots,x_k)\in U.
\end{equation}
Then,
$$
\BV F \geq \frac{c^2}{4^{k+1} k!} \min_{\emptyset\neq J\subset\{1,\hdots,k\}}\inf_{\substack{V\subset U\\ \lambda^k(V)\geq \lambda^k(U)/2^{k+1}}} \lambda^{|J|}(\Pi_J(V)),
$$
where $\Pi_J$ stands for the projection onto the components whose indices belong to $J$.
\end{theorem}

\begin{proof}
We begin with the special case $I_1=\{1,\hdots,k\}$ and $I_2=\emptyset$. For $x_1,\hdots,x_k\in\BX$ and an index set $J=\{j_1,\hdots,j_{|J|}\}\subset \{1,\hdots,k\}$ we put $D^{|J|}_{x_J}F=D^{|J|}_{x_{j_1},\hdots,x_{j_{|J|}}}F$. Now it follows from \eqref{Dsymmetric} that
\begin{align*}
\sum_{\emptyset\neq J\subset\{1,\hdots,k\}}D_{x_J}^{|J|}F & = \sum_{\emptyset\neq J\subset\{1,\hdots,k\}} \sum_{I\subset J} (-1)^{|J|-|I|} f(\eta+\sum_{i\in I}\delta_{x_i})\\
& = \sum_{I\subset \{1,\hdots,k\}} (-1)^{|I|} f(\eta+\sum_{i\in I}\delta_{x_i}) \sum_{I\subset J\subset\{1,\hdots,k\}, J\neq \emptyset} (-1)^{|J|}\\
& = f(\eta+\sum_{i\in\{1,\hdots,k\}}\delta_{x_i})-f(\eta),
\end{align*}
where we have used that the interior alternating sum is zero except for $I=\{1,\hdots,k\}$ and $I=\emptyset$. Combining this with \eqref{eq:lowerbound}, we obtain that, for $(x_1,\hdots,x_k)\in U$,
$$c\leq |\BE[f(\eta+\sum_{i\in\{1,\hdots,k\}}\delta_{x_i})-f(\eta)]|=|\BE\sum_{\emptyset\neq J \subset \{1,\hdots,k\}}D_{x_J}^{|J|}F| \leq \sum_{\emptyset\neq J \subset \{1,\hdots,k\}}| \BE D_{x_J}^{|J|}F|.$$
Hence, there must be a non-empty set $I_0\subset\{1,\hdots,k\}$ and a set $V\subset U$ such that
$$|\BE D^{|I_0|}_{x_{I_0}}F|\geq \frac{c}{2^{k}}, \quad \lambda^k\text{-a.e. }(x_1,\hdots,x_k)\in V, \quad \text{ and } \quad \lambda^k(V)\geq\frac{1}{2^k}\lambda^k(U).$$
Thus, it follows from \eqref{covcha} that
\begin{equation}\label{eq:lowerboundspecialcase}
\BV F\geq \frac{c^2}{4^k k!} \min_{\emptyset\neq J \subset\{1,\hdots,k\}}\inf_{\substack{V\subset U \\ \lambda^k(V)\geq \lambda^k(U)/2^k}} \lambda^{|J|}(\Pi_J(V)).
\end{equation}
For arbitrary $I_1,I_2\subset\{1,\hdots,k\}$ with $I_1\cup I_2=\{1,\hdots,k\}$ we can deduce from \eqref{eq:lowerbound} that there is a set $\tilde{U}\subset U$ such that
$$|\BE[f(\eta+\sum_{i\in I_1}\delta_{x_i})-f(\eta)]|\geq \frac{c}{2}, \quad \lambda^k\text{-a.e. }(x_1,\hdots,x_k)\in \tilde{U},$$
and
$$|\BE[f(\eta+\sum_{i\in I_2}\delta_{x_i})-f(\eta)]|\geq \frac{c}{2}, \quad \lambda^k\text{-a.e. }(x_1,\hdots,x_k)\in U\setminus\tilde{U}.$$
Without loss of generality, we can assume that $\lambda^k(\tilde{U})\geq\lambda^k(U)/2$. Hence, it follows from \eqref{eq:lowerboundspecialcase} that
\begin{align*}
\BV F & \geq \frac{(c/2)^2}{4^k k!} \min_{\emptyset\neq J\subset\{1,\hdots,k\}} \inf_{\substack{V\subset\tilde{U}\\ \lambda^k(V)\geq \lambda^k(\tilde{U})/2^k}} \lambda^{|J|}(\Pi_J(V))\\
& \geq \frac{c^2}{4^{k+1}k!} \min_{\emptyset\neq J\subset\{1,\hdots,k\}} \inf_{\substack{V\subset U\\ \lambda^k(V)\geq \lambda^k(U)/2^{k+1}}} \lambda^{|J|}(\Pi_J(V)),
\end{align*}
which concludes the proof.
\end{proof}

\subsection{The case of Poisson processes in Euclidean space}

Some of our results can be further simplified if we assume that $\BX$ is a subset of $\R^d$. In this case we use the following notation. Recall that $B^d(x,r)$ is a closed ball in $\R^d$ with centre $x$ and radius $r$, $B^d_r=B^d(0,r)$, and $B^d=B^d_1$, and that $\ell_d$ is the Lebesgue measure in $\R^d$. For a compact set $A$ let $\partial A$ be its boundary. We also let $r(A)$ stand for the inradius of a compact convex set $A$, and use the symbol $\kappa_d$ to denote the volume of $B_1^d$.

Throughout this subsection, we assume that $\eta_t$, $t>0$, is the restriction of a stationary Poisson process in $\R^d$ to a measurable set $H\subset\R^d$ whose intensity measure $\lambda_t$ is $t$ times the restriction of $\ell_d$ to $H$. By $\bN_H$ we denote the set of all locally finite point configurations in $H$.

\begin{theorem}\label{thm:VarianceEuclidean}
Let $F\in L^2_{\eta_t}$ and let $f:\bN_H\to\R$ be a representative of $F$. Let $k\in\N$ and $I_1,I_2\subset\{1,\hdots,k\}$ with $I_1\cup I_2=\{1,\hdots,k\}$ and define
$$
g(x_1,\hdots,x_k):=\big|\BE[f(\eta_t+\sum_{i\in I_1}\delta_{x_i})-f(\eta_t+\sum_{i\in I_2}\delta_{x_i})]\big|, \quad (x_1,\hdots,x_k)\in\R^{dk}.
$$
Assume that there are $\hat{x}_1,\hdots,\hat{x}_k\in\R^d$ such that $g$ is continuous in $(\hat{x}_1,\hdots,\hat{x}_k)$ and that there is a constant $c>0$ such that $g(\hat{x}_1,\hdots,\hat{x}_k)\geq c$. Moreover, let $A\subset\R^d$ and $\varepsilon>0$ be such that
$$
g(\hat{x}_1+z,\hat{x}_2+y_2+z,\hdots,\hat{x}_k+y_k+z)=g(\hat{x}_1,\hat{x}_2+y_2,\hdots,\hat{x}_k+y_k)
$$
for all $z\in A$ and $y_2,\hdots,y_k\in B^d_{\varepsilon}$. Then
$$
\tau :=\sup\{r\in(0,\varepsilon): g(\hat{x}_1,\hat{x}_2+y_2,\hdots,\hat{x}_k+y_k)>c/2 \text{ for all } y_2,\hdots,y_k\in B^d_r\}>0
$$
and
$$
\BV F \geq \frac{c^2}{4 \cdot 8^{k+1} k!} \min_{j=1,\hdots,k}  2^{-d(k-j)}(t\kappa_d \tau^d)^{j-1} t \ell_d(A).
$$
\end{theorem}
\begin{proof}
The continuity of $g$ in $(\hat{x}_1,\hdots,\hat{x}_k)$ and the assumption $g(\hat{x}_1,\hdots,\hat{x}_k)\geq c$ ensure that $\tau>0$. Now we define
$$
U:=\{(\hat{x}_1+z,\hat{x}_2+y_2+z,\hdots,\hat{x}_k+y_k+z): z\in A, y_2,\hdots,y_k\in B^d_\tau\}.
$$
Note that $g>c/2$ on $U$. A straightforward computation shows that
\begin{equation}\label{eq:MeasureU}
\lambda_t^k(U)= (t\kappa_d \tau^d)^{k-1} t\ell_d(A).
\end{equation}
In order to apply Theorem \ref{thm:Variancepositive}, we have to compute
$$
\min_{\emptyset\neq J\subset \{1,\hdots,k\}}\inf_{\substack{V\subset U\\ \lambda_t^k(V)\geq \lambda_t^k(U)/2^{k+1}}} \lambda^{|J|}_t(\Pi_J(V)).
$$
Let $\emptyset \neq J\subset\{1,\hdots,k\}$ and let $x_J$ be the components of $x=(x_1,\hdots,x_k)\in \R^{dk}$ whose indices belong to $J$. By definition of $U$, we have that $y_i-y_j\in B^d(\hat{x}_i-\hat{x}_j, 2\tau)$, $i,j\in\{1,\hdots,k\}$, for all $(y_1,\hdots,y_k)\in U$. This means that, for any given $y_J\in \R^{d|J|}$,
$$
\lambda_t^{k-|J|}(\{y_{J^C}\in \R^{d(k-|J|)}: (y_J,y_{J^C})\in U\}) \leq (2^dt \kappa_d \tau^d)^{k-|J|}.
$$
For any $V\subset U$, this provides the second inequality in
\begin{align*}
\lambda_t^k(V) & \leq t^k \int_{(\R^d)^k} \I\{x_J\in \Pi_J(V)\} \I\{(x_1,\hdots,x_k)\in U\} \, \dint x_1 \hdots \dint x_k\\
 & \leq t^{|J|} \int_{(\R^d)^{|J|}} \I\{x_J\in \Pi_J(V)\}  (2^dt \kappa_d \tau^d)^{k-|J|} \, \dint x_J = \lambda^{|J|}_t(\Pi_J(V))  (2^d t \kappa_d \tau^d)^{k-|J|}.
\end{align*}
Consequently, for any $V\subset U$ and $\emptyset \neq J \subset \{1,\hdots,k\}$ we have
$$\lambda_t^{|J|}(\Pi_J(V))\geq (2^dt\kappa_d \tau^d)^{-(k-|J|)} \lambda_t^k(V).$$
Together with \eqref{eq:MeasureU}, we obtain that
\begin{align*}
\min_{\emptyset \neq J \subset \{1,\hdots,k\}} \inf_{\substack{ V\subset U\\ \lambda_t^k(V)\geq \lambda_t^k(U)/2^{k+1}}} \lambda^{|J|}_t(\Pi_J(V)) & \geq \min_{j=1,\hdots,k} (2^dt\kappa_d \tau^d)^{-(k-j)} 2^{-k-1} (t\kappa_d \tau^d)^{k-1} t\ell_d(A)\\
& \geq 2^{-k-1} \min_{j=1,\hdots,k}  2^{-d(k-j)}(t\kappa_d \tau^d)^{j-1} t \ell_d(A),
\end{align*}
which concludes the proof.
\end{proof}

For a family of Poisson functionals $(F_t)_{t\geq 1}$ depending on the intensity of the underlying Poisson process the following corollary ensures that the asymptotic variance does not degenerate.

\begin{corollary}\label{corol:VarianceEuclidean}
Let $F_t\in L^2_{\eta_t}$ with a representative $f_t:\bN_H\to\R$ for $t\geq 1$.  Let $k\in\N$ and $I_1,I_2\subset\{1,\hdots,k\}$ with $I_1\cup I_2=\{1,\hdots,k\}$ and define
$$
g_t(x_1,\hdots,x_k):=\big|\BE[f_t(\eta_t+\sum_{i\in I_1}\delta_{x_i})-f_t(\eta_t+\sum_{i\in I_2}\delta_{x_i})]\big|, \quad (x_1,\hdots,x_k)\in\R^{dk},
$$
for $t\geq 1$. Assume that there are $\hat{x}_1,\hdots,\hat{x}_k\in\R^d$ such that $g_1$ is continuous in $(\hat{x}_1,\hdots,\hat{x}_k)$ and $g_1(\hat{x}_1,\hdots,\hat{x}_k)>0$. Moreover, let $A\subset\R^d$ with $\ell_d(A)>0$ and $\varepsilon>0$ be such that
\begin{align*}
& g_t(\hat{x}_1+z,\hat{x}_1+z+t^{-1/d}(\hat{x}_2-\hat{x}_1+y_2),\hdots,\hat{x}_1+z+t^{-1/d}(\hat{x}_k-\hat{x}_1+y_k))\\
& =g_1(\hat{x}_1,\hat{x}_2+y_2,\hdots,\hat{x}_k+y_k)
\end{align*}
for all $z\in A\cup\{0\}$, $y_2,\hdots,y_k\in B^d_{\varepsilon}$ and $t\geq 1$. Then there is a constant $\sigma>0$ such that $\BV F_t \geq \sigma t$ for $t\geq 1$.
\end{corollary}

\begin{proof}
Define $c:= g_1(\hat{x}_1,\hdots ,\hat{x}_k)$. For any $t\geq 1$ let $\varepsilon_t=t^{-1/d}\varepsilon$ and
\begin{align*}
\tau_t & :=\sup\{r\in(0,\varepsilon_t): g_t(\hat{x}_1,\hat{x}_1+t^{-1/d}(\hat{x}_2-\hat{x}_1)+y_2,\hdots,\hat{x}_1+t^{-1/d} (\hat{x}_k-\hat{x}_1)+y_k)>c/2\\
 & \quad \quad \quad  \text{ for all } y_2,\hdots,y_k\in B^d_r\}\\
& =t^{-1/d}\sup\{r\in(0,\varepsilon): g_1(\hat{x}_1,\hat{x}_2+y_2,\hdots,\hat{x}_k+y_k)>c/2 \text{ for all } y_2,\hdots,y_k\in B^d_r\}\\
& =t^{-1/d} \tau_1.
\end{align*}
Choosing $\hat{x}_i^{(t)}=\hat{x}_1+t^{-1/d}(\hat{x}_i-\hat{x}_1)$, $i\in\{1,\hdots,k\}$, we can apply Theorem \ref{thm:VarianceEuclidean} for every $t\geq 1$, which yields the assertion.
\end{proof}

\section{Stabilizing Poisson functionals}\label{s:stab}

The following result is the main bound used in the geometric applications discussed in Subsection \ref{ss:kn} and Subsection \ref{ss:pvt} and the underlying result of Proposition \ref{introex}.

\begin{theorem}\label{thm:Stabilizing}
Let $F\in \dom D$ with $\BV F>0$ and denote by $N$ a standard Gaussian random variable. Assume that there are constants $c_1,c_2,p_1,p_2>0$ such that
\begin{equation}\label{eqn:Conditionp1}
\BE |D_xF|^{4+p_1} \leq c_1, \quad \lambda\text{-a.e. }x\in\BX,
\end{equation}
and
\begin{equation}\label{eqn:Conditionp2}
\BE |D^2_{x_1,x_2}F|^{4+p_2} \leq c_2, \quad \lambda^2\text{-a.e. }(x_1,x_2)\in\BX^2,
\end{equation}
and let $\overline{c}=\max\{1,c_1,c_2\}$. Then
\begin{align*}
 d_W\bigg(\frac{F-\BE F}{\sqrt{\BV F}},N\bigg)
 & \leq \frac{5\overline{c}}{\BV F} \bigg[ \int  \bigg(\int \BP(D^2_{x_1,x_2}F\neq 0)^{p_2/(16+4p_2)} \,  \lambda(\dint x_2) \bigg)^2 \, \lambda(\dint x_1)\bigg]^{1/2} \\
 & \quad +\frac{\overline{c}}{(\BV F)^{3/2}} \int \BP(D_xF\neq 0)^{(1+p_1)/(4+p_1)} \, \lambda(\dint x)
\end{align*}
and
\begin{align*}
 d_K\bigg(\frac{F-\BE F}{\sqrt{\BV F}},N\bigg)
 & \leq \frac{5\overline{c}}{\BV F} \bigg[ \int  \bigg(\int \BP(D^2_{x_1,x_2}F\neq 0)^{p_2/(16+4p_2)} \,  \lambda(\dint x_2) \bigg)^2 \, \lambda(\dint x_1)\bigg]^{1/2} \\
 & \quad +\frac{\overline{c} \, \Gamma_F^{1/2}}{\BV F} + \frac{2\overline{c} \, \Gamma_F}{(\BV F)^{3/2}} +\frac{\overline{c} \, \Gamma_F^{5/4}+2\overline{c} \,  \Gamma_F^{3/2}}{(\BV F)^2}  \\
& \quad +\frac{\sqrt{6}\overline{c}+\sqrt{3}\overline{c}}{\BV F} \bigg[ \int \BP(D^2_{x_1,x_2}F\neq 0)^{p_2/(8+2p_2)} \, \lambda^2(\dint(x_1,x_2))\bigg]^{1/2}.
\end{align*}
with
$$
\Gamma_F:=\int \BP(D_xF\neq 0)^{p_1/(8+2p_1)} \, \lambda(\dint x).
$$
\end{theorem}

\begin{proof}
For the proof we estimate the right-hand sides of the bounds in Theorem \ref{thm:maindW} and Theorem \ref{thm:mainKolmogorov}. It follows from H\"older's inequality and the assumptions \eqref{eqn:Conditionp1} and \eqref{eqn:Conditionp2} that
\begin{align*}
\BE \big(D_xF\big)^4 & \leq \BP(D_xF\neq 0)^{p_1/(4+p_1)} \big[\BE |D_xF|^{4+p_1} \big]^{4/(4+p_1)}\\
 & \leq c_1^{4/(4+p_1)} \BP(D_xF\neq 0)^{p_1/(4+p_1)},\\
 \BE \big|D_xF\big|^3 & \leq  c_1^{3/(4+p_1)} \BP(D_xF\neq 0)^{(1+p_1)/(4+p_1)}
\end{align*}
for $\lambda$-a.e.\ $x\in\BX$ and
\begin{align*}
\BE \big(D^2_{x_1,x_2}F\big)^4 & \leq \BP(D^2_{x_1,x_2}F\neq 0)^{p_2/(4+p_2)} \big[\BE |D^2_{x_1,x_2}F|^{4+p_2} \big]^{4/(4+p_2)}\\
& \leq c_2^{4/(4+p_2)} \BP(D^2_{x_1,x_2}F\neq 0)^{p_2/(4+p_2)}
\end{align*}
for $\lambda^2$-a.e.\ $(x_1,x_2)\in\BX^2$.
Together with further applications of H\"older's inequality, we obtain that
\begin{align*}
\gamma_1 & \leq \frac{4 c_1^{1/(4+p_1)} c_2^{1/(4+p_2)}}{\BV F} \bigg[\int \big( \BP(D^2_{x_1,x_3}F\neq 0) \BP(D^2_{x_2,x_3}F\neq 0)\big)^{p_2/(16+4p_2)} \, \lambda^3(\dint(x_1,x_2,x_3)) \bigg]^{1/2}\!\!, \allowdisplaybreaks\\
\gamma_2 & \leq \frac{c_2^{2/(4+p_2)}}{\BV F} \bigg[\int \big(\BP(D^2_{x_1,x_3}F\neq 0) \BP(D^2_{x_2,x_3}F\neq 0)\big)^{p_2/(8+2p_2)} \, \lambda^3(\dint(x_1,x_2,x_3)) \bigg]^{1/2}, \allowdisplaybreaks \\
\gamma_3 & \leq \frac{c_1^{3/(4+p_1)}}{(\BV F)^{3/2}} \int \BP(D_xF\neq 0)^{(1+p_1)/(4+p_1)} \, \lambda(\dint x), \allowdisplaybreaks \\
\gamma_4 & \leq \frac{c_1^{3/(4+p_1)}}{2 (\BV F)^{2}} \big[\BE (F-\BE F)^4\big]^{1/4} \int \BP(D_xF\neq 0)^{p_1/(8+2p_1)}  \, \lambda(\dint x), \allowdisplaybreaks \\
\gamma_5 & \leq \frac{c_1^{2/(4+p_1)}}{\BV F} \bigg[\int \BP(D_xF\neq 0)^{p_1/(4+p_1)} \, \lambda(\dint x)\bigg]^{1/2}, \allowdisplaybreaks \\
\gamma_6 & \leq \frac{\sqrt{6}c_1^{1/(4+p_1)} c_2^{1/(4+p_2)}}{\BV F} \bigg[\int \BP(D^2_{x_1,x_2}F\neq 0)^{p_2/(8+2p_2)}  \, \lambda^2(\dint(x_1,x_2)) \bigg]^{1/2} \allowdisplaybreaks \\
& \quad + \frac{\sqrt{3}c_2^{2/(4+p_2)}}{\BV F} \bigg[ \int \BP(D^2_{x_1,x_2}F\neq 0)^{p_2/(4+p_2)}  \, \lambda^2(\dint(x_1,x_2))\bigg]^{1/2}.
\end{align*}
By Lemma \ref{lem:BoundFourthMoment}, we have
\begin{align*}
 \frac{\BE (F-\BE F)^4}{(\BV F)^2} & \leq \max\bigg\{ 256 c_1^{4/(4+p_1)} \bigg[\int \BP(D_xF\neq 0)^{p_1/(8+2p_1)} \, \lambda(\dint x) \bigg]^2/(\BV F)^2,\\
& \hskip 3.5cm4 c_1^{4/(4+p_1)} \int \BP(D_xF\neq 0)^{p_1/(4+p_1)} \, \lambda(\dint x)/(\BV F)^2+2 \bigg\}\\
& \leq \max\bigg\{ 256 c_1^{4/(4+p_1)}\Gamma_F^2/(\BV F)^2,4 c_1^{4/(4+p_1)} \Gamma_F/(\BV F)^2+2 \bigg\}
\end{align*}
so that
$$
\gamma_4 \leq  \frac{c_1^{3/(4+p_1)}}{(\BV F)^{3/2}}  \Gamma_F +\frac{c_1^{4/(4+p_1)}}{(\BV F)^2}  \Gamma_F^{5/4} + \frac{2c_1^{4/(4+p_1)}}{(\BV F)^2}  \Gamma_F^{3/2}.
$$
Combining all these estimates concludes the proof.
\end{proof}

\begin{remark}\label{Remark62} \rm
As discussed in the introduction, Poisson functionals occurring in stochastic geometry are often given in the representation
$$
F=\int h(y,\eta) \, \eta(\dint y)
$$
with a measurable function $h: \BX\times \bN\to\R$. Such a Poisson functional has the first and second order difference operators
$$
D_xF=\int D_xh(y,\eta) \, \eta(\dint y)+h(x,\eta+\delta_x), \quad x\in\BX,
$$
and
$$
D^2_{x_1,x_2}F=\int D^2_{x_1,x_2}h(y,\eta) \, \eta(\dint y)+D_{x_1}h(x_2,\eta+\delta_{x_2})+D_{x_2}h(x_1,\eta+\delta_{x_1}), \quad x_1,x_2\in\BX.
$$
\end{remark}



\begin{proof}[Proof of Proposition \ref{introex}]
The assertion follows by deducing the order in $t$ for all summands in both bounds appearing in the statement of Theorem \ref{thm:Stabilizing}.
\end{proof}

\section{Applications to stochastic geometry}\label{sec:StochasticGeometry}

\subsection{$k$-nearest neighbour graph}\label{ss:kn}

Let $\eta_t$ be a homogeneous Poisson process of intensity $t>0$ in a compact convex observation window $H\subset\R^d$ with interior points. For $k\in\N$ the $k$-nearest neighbour graph is constructed by connecting two distinct points $x,y\in\eta_t$ whenever $x$ is one of the $k$-nearest neighbours of $y$ or $y$ is one of the $k$-nearest neighbours of $x$. In the following, we investigate for $\alpha\geq 0$ the sum $L_t^{(\alpha)}$ of the $\alpha$-th powers of the edge lengths of the $k$-nearest neighbour graph, that is
$$
L_t^{(\alpha)}=\frac{1}{2}\sum_{(x,y)\in\eta^2_{t,\neq}} \I\{\text{$x$ $k$-nearest neighbour of $y$ or $y$ $k$-nearest neighbour of $x$}\} \|x-y\|^\alpha.
$$
Here and in the following we identify simple point processes with their support and denote by $\eta_{t,\neq}^2$ the set of all pairs of distinct points of $\eta_t$. For $\alpha=0$, $L_t^{(\alpha)}$ is the number of edges and for $\alpha=1$ the total edge length. We are in particular interested in the asymptotic behaviour of $L_t^{(\alpha)}$ for $t\to\infty$.

Central limit theorems for the total edge length of the $k$-nearest neighbour graph were studied in \cite{AvramBertsimas1993,BaryshnikovYukich2005,BickelBreiman1983,Penrose07,PenroseYukich2001,PenroseYukich2005}. The first quantitative bound for the Kolmogorov distance of order of  $(\log t)^{1+3/4}t^{-1/4}$ was deduced by Avram and Bertsimas in \cite{AvramBertsimas1993}. This bound was improved to the order of $(\log t)^{3d}t^{-1/2}$ by Penrose and Yukich in \cite{PenroseYukich2005}, and the problem has remained open until now of whether the logarithmic factor in the rate of convergence could be removed at all. As shown in the following statement, the answer is indeed positive.

\begin{theorem}\label{thm:knearest}
Let $N$ be a standard Gaussian random variable. Then there are constants $C_\alpha$, $\alpha\geq 0$, only depending on $k$, $H$ and $\alpha$ such that
$$d_K\left(\frac{L_t^{(\alpha)}-\BE L_t^{(\alpha)}}{\sqrt{\BV L_t^{(\alpha)}}},N\right)\leq C_\alpha t^{-1/2}, \quad t\geq 1.$$
\end{theorem}

We prepare the proof of Theorem \ref{thm:knearest} by the following asymptotic result for the variance of $L_t^{(\alpha)}$. Although it follows from the exact variance asymptotics available in the literature (see \cite[Theorem 6.1]{PenroseYukich2001}, for example), we provide an independent proof, both for the sake of completeness, and in order to illustrate the application of the lower bounds for variances established in Section \ref{sec:Variancepositive}.

\begin{lemma}\label{lem:VarianceNNG}
For any $\alpha\geq 0$ there is a constant $\sigma_{\alpha}>0$ depending on $k$, $H$ and $\alpha$ such that
$$\BV L_t^{(\alpha)} \geq \sigma_{\alpha} t^{1-2\alpha/d}, \quad t\geq 1.$$
\end{lemma}

Throughout the proofs of the previous results we consider the Poisson functionals $F_t=t^{\alpha/d}L_t^{(\alpha)}$. By $l_t^{(\alpha)}: \bN_H\to\R$ we denote a representative of $F_t$. For $x\in\R^d$ and $\mu\in\bN_H$ we denote by $N(x,\mu)$ the $k$-nearest neighbours of $x$ with respect to the points of $\mu$ that are distinct from $x$. A crucial fact for controlling the difference operators of $F_t$ is given in the following Lemma:

\begin{lemma}\label{lem:AddOneNNG}
For $x\in H$ and $\mu\in \bN_H$ let
$$
R(x,\mu)=\max\big\{\sup\{\|z_1-z_2\|: z_1\in\mu, x\in N(z_1,\mu+\delta_x), z_2\in N(z_1,\mu)\}, \sup_{z\in N(x,\mu)}\|z-x\|\big\}.
$$
Then,
$$
D_xl_t^{(\alpha)}(\mu)=D_xl_t^{(\alpha)}\big(\mu\cap B^d(x,3R(x,\mu)) \big).
$$

\end{lemma}
\begin{proof}
Inserting the point $x$ can generate new edges and delete existing edges. The new edges are all emanating from $x$ and are, by definition of $R(x,\mu)$, within $B^d(x,R(x,\mu))$. An edge between two points $z_1,z_2\in\mu$ is deleted if the following situations (i) and (ii) are simultaneously verified: (i) $z_1$ has $x$ as a $k$-nearest neighbour and $z_2$ was a $k$-nearest neighbour of $z_1$ before $x$ was added, or $z_2$ has $x$ as a $k$-nearest neighbour and $z_1$ was a $k$-nearest neighbour of $z_1$ before $x$ was added, and (ii) if $x$ is added, $z_2$ is not a $k$-nearest neighbour of $z_1$ and  $z_1$ is not a $k$-nearest neighbour of $z_2$. Thus, the fact that the edge between $z_1$ and $z_2$ is deleted after adding $x$ only depends on the configuration of the points contained in the set $\big( B^d(z_1,\|z_1-z_2\|)\cup B^d(z_2,\|z_1-z_2\|)\big) \cap \mu$ and $x$, which concludes the proof.
\end{proof}

\begin{proof}[Proof of Lemma \ref{lem:VarianceNNG}]
In the sequel, we will use the fact that there are constants $D_{d,k}$ such that the vertices of a $k$-nearest neighbour graph in $\R^d$ have at most degree $D_{d,k}$ (for an argument for the planar case, which can be generalized to higher dimensions, we refer to the proof of Lemma 6.1 in \cite{PenroseYukich2001}).

Now one can choose $m\in\N$ and $z_1,\hdots,z_m\in\R^d$ with $1/2\leq \|z_i\|\leq 1$, $i\in\{1,\hdots,m\}$, such that
\begin{equation}\label{eq:conditionzs}
\big|\big\{i\in\{1,\hdots,m\}: \|z_i-y\| < \max\{\|y\|, \inf_{x\in \partial B^d}\|y-x\|\} \big\} \big|\geq k+1, \quad y\in B^d,
\end{equation}
and such that all pairwise distances between $z_1,\hdots,z_m$ and the origin are different.

For $x\in\interior(H)$ (where $\interior(H)$ stands for the interior of $H$), $\tau>0$ such that $B^d(x,\tau)\subset H$ and a point configuration $\mu\in\bN_H$ the expression
$$
l_1^{(\alpha)}(\mu+\sum_{i=1}^m \delta_{x+\tau z_i}+\delta_x)-l_1^{(\alpha)}(\mu+\sum_{i=1}^m \delta_{x+\tau z_i})
$$
only depends on the points of $\mu$ that are in $B^d(x,3\tau)$ and is not affected by changes of $\mu$ outside of $B^d(x,3\tau)$. This follows from Lemma \ref{lem:AddOneNNG} since \eqref{eq:conditionzs} implies that $R(x,\mu+\sum_{i=1}^m \delta_{x+\tau z_i})$ is given by the points of $\mu$ in $B^d(x,2\tau)$ and that $R(x,\mu+\sum_{i=1}^m \delta_{x+\tau z_i})\leq \tau$.

If $\mu(B^d(x,\tau))=0$, we obtain
$$
l_1^{(\alpha)}(\mu+\sum_{i=1}^m \delta_{x+\tau z_i}+\delta_x)-l_1^{(\alpha)}(\mu+\sum_{i=1}^m \delta_{x+\tau z_i}) \geq k (\tau/2)^\alpha.
$$
Indeed, adding the point $x$ generates $k$ new edges to points of $z_1,\hdots,z_m$, whose length is at least $\tau/2$, and does not delete any edges since by \eqref{eq:conditionzs} all other points keep their $k$-nearest neighbours. If $\mu(B^d(x,\tau))\neq 0$, we have
$$
|l_1^{(\alpha)}(\mu+\sum_{i=1}^m \delta_{x+\tau z_i}+\delta_x)-l_1^{(\alpha)}(\mu+\sum_{i=1}^m \delta_{x+\tau z_i})| \leq D_{d,k} \tau^\alpha
$$
since the balance of generated and deleted edges originating from the addition of the point $x$ equals at most $D_{d,k}$, and each edge has length at most $\tau$. Consequently, we have
\begin{align*}
& |\BE[l_1^{(\alpha)}(\eta_1+\sum_{i=1}^m \delta_{x+\tau z_i}+\delta_x)-l_1^{(\alpha)}(\eta_1+\sum_{i=1}^m \delta_{x+\tau z_i})]|\\
& \geq \exp(-\kappa_d\tau^d) k(\tau/2)^\alpha -(1-\exp(-\kappa_d\tau^d)) D_{d,k} \tau^\alpha.
\end{align*}
Now it is easy to see that the right-hand side is positive if $\tau\leq \tau_0$ with some $\tau_0>0$. Putting $\hat{x}_1=x$ and $\hat{x}_{i+1}=x+\tilde{\tau} z_i$, $i\in\{1,\hdots,m\}$, with $\tilde{\tau}>0$ sufficiently small, we have $\hat{x}_1,\hdots,\hat{x}_{m+1}\in\interior(H)$ with $\tilde{\tau}/2\leq \|\hat{x}_i-\hat{x}_{1}\|\leq \tilde{\tau}$, $i\in\{2,\hdots,m+1\}$, different pairwise distances, $B^d(\hat{x}_{1},4\tilde{\tau})\subset H$ and
$$
|\BE[l_1^{(\alpha)}(\eta_1+\sum_{i=1}^{m+1} \delta_{\hat{x}_i})-l_1^{(\alpha)}(\eta_1+\sum_{i=2}^{m+1} \delta_{\hat{x}_i})]|>0.
$$
We define $g_t: H^{m+1}\to\R$, $t\geq 1$, by
$$
g_t(x_1,\hdots,x_{m+1}):= |\BE[l_t^{(\alpha)}(\eta_t+\sum_{i=1}^{m+1} \delta_{x_i})-l_t^{(\alpha)}(\eta_t+\sum_{i=2}^{m+1} \delta_{x_i})]|.
$$
Obviously, we have $g_1(\hat{x}_1,\hdots,\hat{x}_{m+1})>0$, and the different pairwise distances imply that $g_1$ is continuous in $(\hat{x}_1,\hdots,\hat{x}_{m+1})$. Since the expectation in the definition of $g_1(\hat{x}_1,\hdots,\hat{x}_{m+1})$ only depends on the points of $\eta_1$ in $B^d(\hat{x}_{m+1},3\tilde{\tau})$ and, by Lemma \ref{lem:AddOneNNG}, such a property still holds if $\hat{x}_1,\hdots,\hat{x}_{m+1}$ are slightly disturbed, there are a set $A\subset\R^d$ with $\ell_d(A)>0$ and a constant $\varepsilon>0$ such that
$$
g_1(\hat{x}_1+z,\hat{x}_2+y_2+z,\hdots,\hat{x}_{m+1}+y_{m+1}+z)=g_1(\hat{x}_1,\hat{x}_2+y_{2},\hdots,\hat{x}_{m+1}+y_{m+1})
$$
for all $z\in A$ and $y_2,\hdots,y_{m+1}\in B^d_\varepsilon$. By the scaling property of a homogeneous Poisson process and the definition of $l_t^{(\alpha)}$ we have
\begin{align*}
& g_t(\hat{x}_1+z,\hat{x}_1+z+t^{-1/d}(\hat{x}_2-\hat{x}_1+y_2),\hdots,\hat{x}_1+z+t^{-1/d}(\hat{x}_{m+1}-\hat{x}_1+y_{m+1}))\\
& =g_1(\hat{x}_1,\hat{x}_2+y_2,\hdots,\hat{x}_{m+1}+y_{m+1})
\end{align*}
for all $z\in A\cup\{0\}$, $y_2,\hdots,y_{m+1}\in B^d_\varepsilon$ and $t\geq 1$. Now Corollary \ref{corol:VarianceEuclidean} concludes the proof.
\end{proof}

Moreover, we will use the following Lemma, which is shown in the proof of Lemma 2.5 in \cite{LastPenrose2013}.

\begin{lemma}\label{lem:Appendix}
Let $H\subset\R^d$ be a compact convex set with non-empty interior. Then there is a constant $c_H>0$ depending on $H$ such that
$$
\ell_d(B^d(x,r)\cap H) \geq c_H r^d
$$
for all $x\in H$ and $0< r \leq \max_{z_1,z_2\in H}\|z_1-z_2\|$.
\end{lemma}

\begin{proof}[Proof of Theorem \ref{thm:knearest}]
We aim at applying Proposition \ref{introex} with $p_1=p_2=1$. Note that, for $x\in H$ and $0<r\leq \max_{z_1,z_2\in H}\|z_1-z_2\|$,
$$
\BP(\eta_t(B^d(x,r))<k) \leq \sum_{i=0}^{k-1}\frac{t^i \kappa_d^i r^{id}}{i!} \exp(-tc_Hr^d) \leq C \exp(-tcr^d)
$$
with suitable constants $C,c>0$, where we have used Lemma \ref{lem:Appendix} in the first inequality.
For $x,y\in H$ this implies that
\begin{equation}\label{eq:Pedge}
\BP(y\in N(x,\eta_t+\delta_y) \text{ or } x\in N(y,\eta_t+\delta_x)) \leq \tilde{C} \exp(-t\tilde{c}\|x-y\|^d)
\end{equation}
with suitable constants $\tilde{C},\tilde{c}>0$.

For $x_1,x_2\in H$ we put $r=\|x_1-x_2\|$. If we assume that
\begin{itemize}
\item [(A)] $\eta_t(B^d(x_1,r/8)\cap H)\geq k+1$,
\item [(B)] $y\in B^d(x_1,r/8)\cap H$ for all $y\in\eta_t$ with $x_1\in N(y,\eta_t+\delta_{x_1})$,
\end{itemize}
we have $R(x_1,\eta_t)\leq r/4$. Under the additional assumption
\begin{itemize}
\item [(C)] $x_1\notin N(x_2,\eta_t+\delta_{x_1})$
\end{itemize}
we see that $R(x_1,\eta_t+\delta_{x_2})=R(x_1,\eta_t)$. Consequently, Lemma \ref{lem:AddOneNNG} implies that $D^2_{x_1,x_2}F_t=0$ if the conditions (A)-(C) are satisfied. Obviously, we have
$$
\BP(x_1\in N(x_2,\eta_t+\delta_{x_1})) \leq \tilde{C} \exp(-t \tilde{c} r^d)
$$
and
$$
\BP(\eta_t(B^d(x_1,r/8)\cap H)< k+1) \leq \overline{C} \exp(-t \overline{c} r^d)
$$
with suitable constants $\overline{C},\overline{c}>0$. Using the Mecke formula, \eqref{eq:Pedge} and spherical coordinates, we see that
\begin{align*}
 \BP(\exists y\in \eta_t \setminus B^d(x_1,r/8): x_1\in N(y,\eta_t+\delta_{x_1}))  & \leq \BE \sum_{y\in\eta_t\setminus B^d(x_1,r/8)} \I(x_1\in N(y,\eta_t+\delta_{x_1}))\\
 & \hskip -3cm \leq t \int_{\R^d \setminus B^d(x_1,r/8)} \BP(x_1\in N(y,\eta_t+\delta_{x_1})) \, \dint y\\
 & \hskip -3cm \leq t \int_{\R^d \setminus B^d(x_1,r/8)} \tilde{C} \exp(-t \tilde{c} \|y-x_1\|^d) \, \dint y \leq \hat{C} \exp(-t\hat{c} r^d)
\end{align*}
with suitable constants $\hat{C},\hat{c}>0$. Altogether, we see that there are constants $C^*,c^*>0$ such that
$$
\BP(D^2_{x_1,x_2}F_t\neq 0) \leq C^* \exp(-tc^* \|x_1-x_2\|^d).
$$
Combining this with spherical coordinates shows that
$$
\sup_{x\in H,\ t\geq 1} t \int \BP(D^2_{x,y}F_t\neq 0)^{1/20} \, \dint y <\infty.
$$

Let $M_1(x,\eta_t)$ be the length of the longest edge that is generated by adding the point $x$ and let $M_2(x,\eta_t)$ be the length of the longest edge that is removed by adding the point $x$. It follows from the Mecke formula and \eqref{eq:Pedge} that
\begin{align*}
\BP(M_1(x,\eta_t)\geq s) & \leq \BE \sum_{y\in\eta_t\setminus B^d(x,s)} \I\{y\in N(x,\eta_t) \text{ or } x\in N(y,\eta_t+\delta_x)\}\\
& \leq \tilde{C} t \int_{\R^d\setminus B^d(x,s)} \exp(-t\tilde{c}\|x-y\|^d) \, \dint y \leq \tilde{C}_1 \exp(-t \tilde{c}_1 s^d)
\end{align*}
and
\begin{align*}
& \BP(M_2(x,\eta_t)\geq s)\\
& \leq \BE \sum_{(y_1,y_2)\in\eta^2_{t,\neq}} \I\{\|y_1-y_2\|\geq s,y_2\in N(y_1,\eta_t), x\in N(y_1,\eta_t+\delta_{x}) \} \allowdisplaybreaks\\
& \leq t^2 \int \I\{\|y_1-y_2\|\geq s\} \BP(x\in N(y_1,\eta_t+\delta_x),y_2\in N(y_1,\eta_t+\delta_{y_2})) \, \dint(y_1,y_2) \allowdisplaybreaks\\
& \leq t^2 \int \I\{\|y_1-y_2\|\geq s\} \BP(x\in N(y_1,\eta_t+\delta_x))^{1/2} \BP(y_2\in N(y_1,\eta_t+\delta_{y_2}))^{1/2} \, \dint(y_1,y_2) \allowdisplaybreaks\\
& \leq t^2 \tilde{C}^2 \int  \int_{\R^d\setminus B^d(y_1,s)} \exp(-t \tilde{c} \|y_1-y_2\|^d/2) \, \dint y_2 \, \exp(-t\tilde{c}\|y_1-x\|^d/2) \, \dint y_1\\
& \leq \tilde{C}_2 \exp(-t \tilde{c}_2 s^d)
\end{align*}
with suitable constants $\tilde{C}_1,\tilde{c}_1,\tilde{C}_2,\tilde{c}_2>0$. In a similar way we obtain that
$$
\BP(M_1(x_1,\eta_t+\delta_{x_2})\geq s) \leq \tilde{C}_3 \exp(-t\tilde{c}_3 s^d) \ \text{ and } \ \BP(M_2(x_1,\eta_t+\delta_{x_2})\geq s) \leq \tilde{C}_4 \exp(-t\tilde{c}_4 s^d)
$$
for all $x_1,x_2\in H$ with constants $\tilde{C}_3,\tilde{c}_3,\tilde{C}_4,\tilde{c}_4>0$. Since at most $D_{d,k}$ edges are generated and removed by  adding the point $x\in H$, we have
$$
|D_xF_t|\leq  D_{d,k} t^{\alpha/d} \max\{M_1(x,\eta_t),M_2(x,\eta_t)\}^{\alpha}
$$
and, for $x_1,x_2\in H$,
$$
|D^2_{x_1,x_2}F_t| \leq 2 D_{d,k} t^{\alpha/d} \max\{M_1(x_1,\eta_t),M_2(x_1,\eta_t),M_1(x_1,\eta_t+\delta_{x_2}),M_2(x_1,\eta_t+\delta_{x_2})\}^{\alpha}.
$$
Because of the exponential tail probabilities for the expressions in the maxima, there are constants $c_1$ and $c_2$ such that
$$
\BE |D_xF_t|^5 \leq c_1, \quad x\in H, \quad \text{and} \quad \BE|D^2_{x_1,x_2}F_t|^5 \leq c_2, \quad x_1,x_2\in H.
$$
Now Proposition \ref{introex} concludes the proof.
\end{proof}

\subsection{Poisson-Voronoi tessellation}\label{ss:pvt}

Let $\eta_t$ be a stationary Poisson process in $\R^d$ whose intensity
measure is $t$ times the Lebesgue measure. Now we can divide the whole $\R^d$ into cells
$$
C(x,\eta_t)=\{y\in\R^d: \|x-y\|\leq \|z-y\|, z\in\eta_t\}, \quad x\in\eta_t,
$$
i.e.\ the cell with nucleus $x$ contains all points of $\R^d$ such that $x$
is the closest point of $\eta_t$. The collection of all these cells is called
{\em Poisson-Voronoi tessellation}. All its cells are (almost surely) bounded polytopes,
and we let $X^k_t$, $k\in\{0,\ldots,d\}$,
denote the system of all $k$-faces  of these polytopes.
For an introduction to some fundamental mathematical properties
of such tessellations, as well as for relevant definitions, see \cite[Chapter 10]{SW08}.

Let $H$ be a compact convex set with non-empty interior. We are interested in the normal approximation of the Poisson functionals
\begin{align*}
V^{(k,i)}_t:=\sum_{G\in X^k_t}V_i(G\cap H),
\end{align*}
where $k\in\{0,\ldots,d\}$, $i\in\{0,\hdots,\min\{k,d-1\}\}$ and $V_i(\cdot)$ is the $i$-th {\em intrinsic volume}
(see \cite{SW08}).
In particular $V^{(d-1,d-1)}_t$ is the total surface content (edge length in case
$d=2$) of all cells within $H$, while $V^{(k,0)}_t$ is the total
number of all $k$-faces intersecting $H$. We do not allow $k=i=d$ since $V_t^{(d,d)}=\ell_d(H)$ is constant.

Central limit theorems for the functionals $V^{(k,i)}_t$ are implied
by the mixing properties of the Poisson Voronoi tessellation
derived by Heinrich in \cite{Heinrich1994}. The Voronoi tessellation within the observation window can be also constructed with respect to a finite Poisson process in the observation and not with respect to a stationary Poisson process. For this slightly different situation, which has the same asymptotic behaviour as the setting described above, central limit theorems were derived by stabilization techniques in \cite{AvramBertsimas1993,BaryshnikovYukich2005,Penrose07,PenroseYukich2001,PenroseYukich2005}. Quantitative bounds on the
Kolmogorov distance for the edge length in the planar case were proved by Avram and Bertsimas \cite{AvramBertsimas1993}
and improved by Penrose and Yukich in \cite{PenroseYukich2005}. These bounds of the orders of $(\log t)^{1+3/4}t^{-1/4}$ and $(\log t)^{3d}t^{-1/2}$, respectively, can be further improved as the following theorem shows.

\begin{theorem}\label{thm:PVTessellation}
Let $N$ be a standard Gaussian random variable. Then there are constants
$c_{i,k}$, $k\in\{0,\hdots,d\}$, $i\in\{0,\hdots,\min\{k,d-1\}\}$, such that
$$
d_K\bigg(\frac{V^{(k,i)}_t-\BE V^{(k,i)}_t}{\sqrt{\BV V^{(k,i)}_t}} ,N\bigg)
\leq c_{k,i} t^{-1/2}, \quad t\geq 1.
$$
\end{theorem}

Let $k\in\{0,\hdots,d\}$ and $i\in\{0,\hdots,\min\{k,d-1\}\}$ be fixed in the following. In order to prove the previous theorem we consider the Poisson functionals
$\tilde{V}_t^{(k,i)}=t^{i/d} V_t^{(k,i)}$ and denote by $\tilde{v}_t^{(k,i)}:\bN_{\R^d}\to\R$
a representative of $\tilde{V}_t^{(k,i)}$. We start by proving the following
lemma for the variance. More details on the asymptotic covariance structure
of these random variables are provided in the recent preprint
\cite{LaOchs13}.

\begin{lemma}\label{lem:VariancePV}
There are constants $\sigma_{k,i}>0$, $k\in\{0,\hdots,d\}$, $i\in\{0,\hdots,\min\{k,d-1\}\}$, such that
$$
\BV V_t^{(k,i)} \geq \sigma_{k,i} t^{1-2i/d}, \quad t\geq 1.
$$
\end{lemma}

\begin{proof}
Let $\hat{x}_1$ be in the interior of $H$ and let $\varepsilon=\inf_{y\in \partial H} \|\hat{x}_1-y\|$ and $H_{\varepsilon}=\{x\in H: \inf_{y\in\partial H}\|y-x\|\geq\varepsilon/2\}$. Now we choose $l\in\N$ and points $\hat{x}_2,\hdots,\hat{x}_l\in B^{d}(\hat{x}_1,\varepsilon/2)$ such that
$$
\sup_{y\in \partial B^{d}(\hat{x}_1,\varepsilon/2)} \min_{i=2,\hdots,l} \|y-\hat{x}_i\|< \frac{\varepsilon}{4}
$$
and $\hat{x}_1,\hdots,\hat{x}_l$ are in general position (see \cite[p.\ 472]{SW08}).
Now we can choose a point $\hat{x}_{l+1}\in B^d(\hat{x}_1,\varepsilon/2)$ such that $\hat{x}_1,\hdots,\hat{x}_{l+1}$ are still in general position and such that
\begin{equation}\label{eq:Conditiongpositive}
\BE[\tilde{v}_1^{(k,i)}(\eta_1+\sum_{i=1}^{l+1} \delta_{\hat{x}_i})-\tilde{v}_1^{(k,i)}(\eta_1+\sum_{i=1}^{l} \delta_{\hat{x}_i})]>0,
\end{equation}
as can be seen from the following argument. For an arbitrary $w\in\R^d$ with $\|w\|=1$ we have almost surely that
$$
\lim_{r\to 0}\tilde{v}^{(k,i)}_1(\eta_1+\sum_{i=1}^l \delta_{\hat{x}_i}+\delta_{\hat{x}_1+r w})-\tilde{v}^{(k,i)}_1(\eta_1+\sum_{i=1}^l \delta_{\hat{x}_i})>0.
$$
This is the case since adding $\hat{x}_1+r w$ for sufficiently small $r$ means that we split the cell around $\hat{x}_1$ in two cells which generates new faces, whereas the old faces are slightly moved (here we use the fact that the points of $\eta_1$ and the additional points are in general position almost surely). Now the dominated convergence theorem implies the same for the expectations. Putting $\hat{x}_{l+1}=\hat{x}_1+r w$ with $w$ such that $\hat{x}_1,\hdots,\hat{x}_{l+1}$ are in general position and $r$ sufficiently small yields \eqref{eq:Conditiongpositive}. We define $g_t: H^{l+1}\to\R$, $t\geq 1$, by
$$
g_t(x_1,\hdots,x_{l+1}):= \BE[\tilde{v}_t^{(k,i)}(\eta_t+\sum_{i=1}^{l+1} \delta_{x_i})-\tilde{v}_t^{(k,i)}(\eta_t+\sum_{i=1}^l \delta_{x_i})].
$$
By \eqref{eq:Conditiongpositive} we have that $g_1(\hat{x}_1,\hdots,\hat{x}_{l+1})>0$. Since the points $\hat{x}_1,\hdots,\hat{x}_{l+1}$ are by choice in general position, $g_1$ is continuous in $(\hat{x}_1,\hdots,\hat{x}_{l+1})$.

For $y_2,\hdots,y_{l+1}\in B^d_{\varepsilon/4}$, $z\in H_{\varepsilon}-\hat{x}_1$ and $t\geq 1$ we have that
\begin{align*}
& g_t(\hat{x}_1+z,\hat{x}_1+z+t^{-1/d}(\hat{x}_2+y_2-\hat{x}_1),\hdots,\hat{x}_1+z+t^{-1/d}(\hat{x}_{l+1}+y_{l+1}-\hat{x}_1))\\
& =g_1(\hat{x}_1,\hat{x}_2+y_2,\hdots,\hat{x}_{l+1}+y_{l+1}),
\end{align*}
which follows from the stationarity of $\eta_t$, the scaling property $\eta_t\overset{d}{=}t^{-1/d} \eta_1$ and the $i$-homogeneity of the $i$-th intrinsic volume. Moreover, we have used that on both sides the cell around $\hat{x}_1+z$ is included in $H$, which is a consequence of the construction of $\hat{x}_1,\hdots,\hat{x}_{l+1}$ and of the choice of $y_2,\hdots,y_{l+1}$ and $z$. Now Corollary \ref{corol:VarianceEuclidean} yields the assertion.
\end{proof}

For $\mu\in\bN_{\R^d}$ and $x\in\R^d$ we denote by $R(x,\mu)$ the maximal distance of a vertex of $C(x,\mu+\delta_x)$ to $x$. We define $d(x,A)=\inf_{z\in A}\|x-z\|$ for $x\in\R^d$ and $A\subset\R^d$.

\begin{lemma}\label{lem:BoundsVoronoi}
There are constants $\tilde{C_1},\tilde{C}_2,\tilde{C}_3,\tilde{c}_1,\tilde{c}_2,\tilde{c}_3>0$ such that
\begin{equation}\label{eq:BoundPCircumradius}
\BP(R(x,\eta_t) \geq s) \leq \tilde{C}_1 \exp(-t\tilde{c}_1 s^d), \quad s\geq 0, \quad x\in\R^d, \quad t\geq 1,
\end{equation}
\begin{equation}\label{eq:BoundDistanceW}
\BP(C(x,\eta_t+\delta_x)\cap H \neq \emptyset) \leq \tilde{C}_2 \exp(-t \tilde{c}_2 d(x,H)^d), \quad x\in \R^d, \quad t\geq 1,
\end{equation}
and
$$
\BP(C(x_1,\eta_t+\delta_{x_1})\cap C(x_2,\eta_t+\delta_{x_2}) \neq \emptyset) \leq \tilde{C}_3 \exp(-t \tilde{c}_3 \|x_1-x_2\|^d), \quad x_1,x_2\in\R^d, \quad t\geq 1.
$$
\end{lemma}

\begin{proof}
The inequality \eqref{eq:BoundPCircumradius} follows from Theorem 2 in \cite{HugSchneider2007}. The other bounds can be deduced from \eqref{eq:BoundPCircumradius}.
\end{proof}

\begin{proof}[Proof of Theorem \ref{thm:PVTessellation}]
For $x\in\R^d$ and $\mu\in\bN_{\R^d}$ let $A_{x,\mu}=\{y\in\mu: C(x,\mu+\delta_x)\cap C(y,\mu)\neq \emptyset\}$, which is the set of all neighbour points of the cell around $x$. It is easy to see that all these points must be included in $B^d(x,2R(x,\mu))$ so that
$$|A_{x,\mu}|\leq \mu(B^d(x,2R(x,\mu))).$$
By adding the point $x$ to $\mu$, some $k$-faces of the Voronoi tessellation are changed or removed. Since each of these faces is associated with $d+1-k$ neighbours of $x$ (because the tessellation is {\em normal}, see \cite[Theorem 10.2.3]{SW08}), at most $\mu(B^d(x,2R(x,\mu)))^{d+1-k}$ $k$-faces are changed or removed. By the monotonicity of the intrinsic volumes the $i$-th intrinsic volume of each of these $k$-faces is reduced by $V_i(B^d(x,R(x,\mu)))$ at most. On the other hand, adding the point $x$ generates some new $k$-faces. Each of them is associated with $d-k$ neighbours of $x$ and their $i$-th intrinsic volumes are bounded by $V_i(B^d(x,R(x,\mu)))$.
Altogether we see that, for $x\in\R^d$,
\begin{align*}
|D_{x}\tilde{v}_t^{(k,i)}(\mu)| & \leq t^{i/d} V_i(B^d(x,R(x,\mu))) \mu(B^d(x,2R(x,\mu)))^{d+1-k}\\
& = t^{i/d} V_i(B^d(0,1)) R(x,\mu)^{i} \mu(B^d(x,2R(x,\mu)))^{d+1-k}.
\end{align*}
By iterating this argument and using the monotonicity of $R(x_1,\mu)$, we see that, for $x_1,x_2\in\R^d$,
\begin{align*}
|D_{x_1}\tilde{v}_t^{(k,i)}(\mu+\delta_{x_2})| & \leq t^{i/d} V_i(B^d(0,1)) R(x_1,\mu+\delta_{x_2})^{i} (\mu+\delta_{x_2})(B^d(x,2R(x_1,\mu+\delta_{x_2})))^{d+1-k}\\
& \leq t^{i/d} V_i(B^d(0,1)) R(x_1,\mu)^{i} (\mu(B^d(x,2R(x_1,\mu)))+1)^{d+1-k}.
\end{align*}
This implies that
$$
|D^2_{x_1,x_2}\tilde{v}_t^{(k,i)}(\mu)| \leq 2 t^{i/d} V_i(B^d(0,1)) R(x_1,\mu)^{i} (\mu(B^d(x,2R(x_1,\mu)))+1)^{d+1-k}, \quad x_1,x_2\in\R^d.
$$
Together with the stationarity of $\eta_t$, we obtain that the fifth moments of $|D_x\tilde{V}_t^{(k,i)}|$ and $|D_{x_1,x_2}^2\tilde{V}_t^{(k,i)}|$ are bounded by linear combinations of the expectations
$$
\BE t^{5i/d} R(0,\eta_t)^{5i} \sum_{(y_1,\hdots,y_{m})\in\eta_{t,\neq}^m} \I\{y_1,\hdots,y_{m}\in B^d(0,2R(0,\eta_t))\}, \quad m\in\{0,\hdots,5d+5-5k\}.
$$
Using the Mecke formula and the monotonicity of $R(0,\mu)$, we see that the letter expression can be bounded by
\begin{align*}
& t^{m+5i/d} \int_{(\R^d)^m} \BE R(0,\eta_t+\delta_{y_1}+\hdots+\delta_{y_m})^{5i}\\
& \hskip 4cm \I\{y_1,\hdots,y_{m}\in B^d(0,2R(0,\eta_t+\delta_{y_1}+\hdots+\delta_{y_m}))\} \, \dint(y_1,\hdots,y_m)\\
& \leq t^{m+5i/d} \int_{(\R^d)^m} \BE R(0,\eta_t)^{5i} \I\{y_1,\hdots,y_{m}\in B^d(0,2R(0,\eta_t))\} \, \dint(y_1,\hdots,y_m)\\
& = 2^{dm} \kappa_d^m t^{m+5i/d} \BE R(0,\eta_t)^{md+5i}.
\end{align*}
Now \eqref{eq:BoundPCircumradius} yields that the right-hand side is uniformly bounded for $t\geq 1$, whence the fifth absolute moments of the first and the second difference operator are also uniformly bounded for $t\geq 1$.

Let $B_H$ and $R_H$ be the circumball and the circumradius of $H$, respectively. We have that $D_x\tilde{V}_t^{(k,i)}= 0$ if $C(x,\eta_t+\delta_{x})\cap H=\emptyset$ since in this case the tessellation in $H$ is the same for $\eta_t$ and $\eta_t+\delta_x$. Together with \eqref{eq:BoundDistanceW} and spherical coordinates, we see that
\begin{align*}
t \int \BP(D_x\tilde{V}_t^{(k,i)}\neq 0)^{1/20} \, \dint x & \leq t \kappa_d R_H^d + t \int_{\R^d\setminus B_H} \tilde{C}_2^{1/20} \exp(-t \tilde{c}_2 d(x,H)^d/20) \, \dint x \\
& \leq t \kappa_d R_H^d + d \kappa_d t \int_{R_H}^\infty \tilde{C}_2^{1/20} \exp(-t \tilde{c}_2 (r-R_H)^d/20) \, r^{d-1}\, \dint r.
\end{align*}
Here, the right-hand side is bounded by  a constant times $t$ for $t\geq 1$.

Observe that $D^2_{x,y}\tilde{V}_t^{(k,i)}=0$ if $C(x,\eta_t+\delta_{x})\cap H=\emptyset$ or $C(x,\eta_t+\delta_{x})\cap C(y,\eta_t+\delta_{y})=\emptyset$ since in both cases we have $D_y\tilde{v}_t^{(k,i)}(\eta_t)=D_y\tilde{v}_t^{(k,i)}(\eta_t+\delta_x)$. For $x\in\R^d$ combining this with Lemma \ref{lem:BoundsVoronoi} implies that
\begin{align*}
 t \int \BP(D^2_{x,y}\tilde{V}_t^{(k,i)}\neq 0)^{1/20} \, \dint y & \leq \tilde{C}_2^{1/20} t \kappa_d d(x,H)^d \exp(-t\tilde{c}_2 d(x,H)^d/20)\\
 & \quad +\tilde{C}_3^{1/20} t \int_{\R^d\setminus B^d(x,d(x,H))} \!\!\!\! \exp(-t \tilde{c}_3 \|x-y\|^d/20) \, \dint y.
\end{align*}
By using polar coordinates and estimating the first summand, we obtain that there are constants $\hat{C},\hat{c}>0$ such that
$$
t \int \BP(D^2_{x,y}\tilde{V}_t^{(k,i)}\neq 0)^{1/20} \, \dint y \leq \hat{C} \exp(-t \hat{c} d(x,H)^d).
$$
Now a similar calculation as above shows that
$$
t \int \bigg(t \int \BP(D^2_{x_1,x_2}\tilde{V}_t^{(k,i)}\neq 0)^{1/20} \, \dint x_2\bigg)^2 \, \dint x_1 \quad \text{and} \quad
t^2 \int \BP(D^2_{x_1,x_2}\tilde{V}_t^{(k,i)}\neq 0)^{1/20} \, \dint(x_1,x_2)
$$
are of order $t$. Now Theorem \ref{thm:Stabilizing} with $p_1=p_2=1$ and Lemma \ref{lem:VariancePV} conclude the proof.
\end{proof}

\section{Functionals of Poisson shot noise random fields}\label{sec:ShotNoise}

We will now describe a further application of our results, dealing with non-linear functionals of stochastic functions that are obtained as integrals of a deterministic kernel with respect to a Poisson measure. As anticipated in the Introduction, when specialised to the case of moving averages (see Section 8.2) our results provide substantial extensions of the findings contained in \cite{EichelsbacherThaele2013, PSTU10, PeccatiZheng2010}, which only considered linear and quadratic functionals.

\subsection{General results}

In this section we consider non-linear functionals of first order Wiener-It\^o integrals depending on a parameter $t\in \mathbb{Y}$. For this purpose, we fix a measurable space $(\mathbb{X}, \mathcal{X})$, as well as a Poisson measure on $\mathbb{X}$ with $\sigma$-finite intensity $\lambda$. We let $(\mathbb{Y},\mathcal{Y})$ be a measurable space and let $f_t: \BX\to\R$, $t\in\mathbb{Y}$, be a family of functions such that $(x,t)\to f_t(x)$ is jointly measurable and
\begin{equation}\label{eq:Momentsf}
\int |f_t(x)|^p \, \lambda(\dint x)<\infty, \quad p\in\{1,2\}, \quad t\in\mathbb{Y}.
\end{equation}
We consider the random field $(X_t)_{t\in\mathbb{Y}}$ defined by $X_t=I_1(f_t)$, $t\in\mathbb{Y}$, where $I_1$ indicates the Wiener-It\^o integral with respect to $\hat{\eta} = \eta-\lambda$.
Using the pathwise representation
$$
X_t=\int f_t(x) \, \eta(\dint x) -\int f_t(x) \, \lambda(\dint x),
$$
we see that $(\omega,t)\to X_t(\omega)$ can be assumed to be jointly measurable. We are interested in the normal approximation of the random variable
\begin{equation}\label{eq:DefinitionF}
F=\int \varphi(X_t) \, \varrho(\dint t),
\end{equation}
where $\varphi:\R\to\R$ is a twice differentiable function and $\varrho$ is a finite measure on $\mathbb{Y}$. For obvious reasons, we require that $\lambda$, $f_t$ and $\varphi$ are such that $F$ is not almost surely constant.

A random variable of the type \eqref{eq:DefinitionF} is the quintessential example of a non-linear functional of the field $(X_t)_{t\in\BY}$. Such fields are crucial for applications, for instance: when $(\mathbb{X}, \mathcal{X}) =(\mathbb{Y}, \mathcal{Y}) = (\R^d, \mathcal{B}(\R^d)) $, $f_t(x) = e^{{\bf i}\langle t,x \rangle}$ (where $\langle\cdot,\cdot\rangle$ is the Euclidean scalar product and ${\bf i}^2=-1$) and $\eta-\lambda$ is adequately complexified, then the field $(X_t)_{t\in\BY}$ represents the prototypical example of a centred stationary field on $\R^d$ (see e.g. \cite[Section 5.4]{AT}, \cite[Section 5.3 and Section 5.4]{MP}); when $(\mathbb{X}, \mathcal{X}) = (\R^2, \mathcal{B}(\R^2))$, $(\mathbb{Y}, \mathcal{Y}) =(\R_+, \mathcal{B}(\R_+))$ and $f_t(u,x) = uf(t-x)$, then $(X_t)_{t\in\BY}$ is a so-called {\em moving average L\'evy process}. Moving average L\'evy processes have gained momentum in a number of fields: for example, starting from the path-breaking paper \cite{BNS}, they have become relevant for the mathematical modeling of stochastic volatility in continuous-time financial models; they are also used in nonparametric Bayesian survival analysis (where they play the role of random hazard rates, see e.g.\ \cite{DBPP, PePrU}). For some recent applications of CLTs involving linear and quadratic functionals of moving average L\'evy processes in a statistical context, see e.g.\ \cite{CL} and the references therein. We refer the reader to \cite{ATW, SpodBook, SpodSur} for a survey of recent examples and applications of limit theorems for non-linear functionals of random fields, as well as to \cite{DMP} for a collection of CLTs involving functionals of a Poisson field on the sphere, with applications to cosmological data analysis.

We assume that
\begin{align}\label{eq:boundphi0}
|\varphi(r)| & \leq h(r), \quad r\in\R,\\
 \label{eq:boundphi1}
|\varphi'(r_1+r_2)| &\leq h(r_1)+h(r_2), \quad r_1,r_2\in\R,\\
\label{eq:boundphi2}
|\varphi''(r_1+r_2+r_3)| & \leq h(r_1)+h(r_2)+h(r_3), \quad r_1,r_2,r_3\in\R,
\end{align}
with a measurable function $h:\R\to[0,\infty)$. We also assume that
\begin{equation}\label{eq:C2}
C_2:=\max\bigg\{\sup_{t\in\mathbb{Y}} \BE h(X_t)^4,1\bigg\}<\infty
\end{equation}
and
\begin{equation}\label{eq:boundpsi}
\iint \psi_t(x)^2+\psi_t(x)^4 \, \lambda(\dint x) \, \varrho(\dint t)<\infty,
\end{equation}
where
$$
\psi_t(x)=C_2^{1/4} |f_t(x)|+H(f_t(x))), \quad x\in\BX, \quad t\in\BY,
$$
and
$$
H(r):=\I\{r\leq 0\} \int_r^0 h(s) \, \dint s + \I\{r>0\} \int_0^r h(s) \, \dint s, \quad r\in\R.
$$
For $g_1,g_2\in L^2(\lambda)$ we write $\langle g_1,g_2\rangle:=\int g_1 g_2 \, \dint \lambda$.

\begin{theorem}\label{thm:PoissonShotNoise}
Let $F$ be given by \eqref{eq:DefinitionF} and assume that \eqref{eq:Momentsf} and \eqref{eq:boundphi0}--\eqref{eq:boundpsi} hold and let $N$ be a standard Gaussian random variable. Then, $F\in L^2_\eta$ and
\begin{align*}
& d_K\bigg(\frac{F-\BE F}{\sqrt{\BV F}},N\bigg)\\
 & \leq \frac{36}{\BV F} \bigg[\int \langle \psi_{t_1}, \psi_{t_3}\rangle \langle \psi_{t_2}, \psi_{t_4}\rangle \langle \psi_{t_3},\psi_{t_4}\rangle + \langle \psi_{t_1}, \psi_{t_2} \rangle \langle \psi_{t_3}, \psi_{t_4}\rangle \langle \psi_{t_1} \psi_{t_2}, \psi_{t_3} \psi_{t_4}\rangle \allowdisplaybreaks\\
 & \hskip 2cm +\langle \psi_{t_1}\psi_{t_2}, \psi_{t_3}\psi_{t_4} \rangle \langle \psi_{t_3},\psi_{t_4} \rangle + \langle \psi_{t_1}\psi_{t_2}, \psi_{t_3}\psi_{t_4} \rangle^2 \, \varrho^4(\dint(t_1,t_2,t_3,t_4)) \bigg]^{1/2} \allowdisplaybreaks \allowdisplaybreaks\\
& \quad + \bigg( \int \langle \psi_{t_1}, \psi_{t_2} \rangle \langle \psi_{t_3}, \psi_{t_4} \rangle + \langle \psi_{t_1}\psi_{t_2}, \psi_{t_3}\psi_{t_4}\rangle \, \varrho^4(\dint(t_1,t_2,t_3,t_4))/(\BV F)^2+2 \bigg)^{1/4}\\
& \quad \quad \frac{16}{(\BV F)^{3/2}}\int \langle \psi_{t_1}\psi_{t_2},\psi_{t_3}\rangle \, \varrho^3(\dint(t_1,t_2,t_3)) \allowdisplaybreaks\\
& \quad + \frac{2\sqrt{2}}{\BV F} \bigg[\int \langle \psi_{t_1}\psi_{t_2},\psi_{t_3}\psi_{t_4}\rangle \, \varrho^4(\dint(t_1,t_2,t_3,t_4))\bigg]^{1/2}.
\end{align*}
\end{theorem}

\begin{proof} We can of course assume that the right-hand side of the inequality in the statement is finite (otherwise, there is nothing to prove). Combining \eqref{eq:boundphi0} and \eqref{eq:C2} with the Cauchy-Schwarz inequality, yields that $\BE\int |\varphi(X_t)| \, \varrho(\dint t) <\infty$ and $F\in L^2_\eta$ (recall that $\varrho$ is finite). Moreover, the subsequent calculations show that, for $\lambda^2$-a.e.\ $(x_1,x_2)\in\BX^2$,
$$
\BE \int |\varphi(X_t+f_t(x_1))|+|\varphi(X_t+f_t(x_1)+f_t(x_2))| \, \varrho(\dint t)<\infty.
$$
Therefore we obtain from \eqref{eq:Momentsf} that, $\BP$-a.s.\ and for $\lambda^2$-a.e.\ $(x_1,x_2)\in \BX^2$,
$$
D_{x_1}F = \int \varphi(X_t+f_t(x_1)) -\varphi(X_t)  \, \varrho(\dint t) = \int \int_0^{f_t(x_1)} \varphi'(X_t+a) \, \dint a \, \varrho(\dint t)
$$
and
\begin{align*}
D^2_{x_1,x_2}F & =\int \varphi(X_t+f_t(x_1)+f_t(x_2)) - \varphi(X_t+f_t(x_1)) - \varphi(X_t+f_t(x_2)) + \varphi(X_t) \, \varrho(\dint t) \notag \\
& = \int \int_0^{f_t(x_1)} \int_0^{f_t(x_2)} \varphi''(X_t+a+b) \, \dint a \, \dint b \, \varrho(\dint t).
\end{align*}
Now \eqref{eq:boundphi1} and \eqref{eq:boundphi2} imply that, for $\lambda^2$-a.e.\ $(x_1,x_2)\in\BX^2$,
$$
|D_{x_1}F| \leq \int h(X_t) |f_t(x_1)| + H(f_t(x_1)) \, \varrho(\dint t)
$$
and
\begin{align*}
|D^2_{x_1,x_2}F| 
\leq \int h(X_t) |f_t(x_1)| \, |f_t(x_2)| + |f_t(x_2)| H(f_t(x_1)) +|f_t(x_1)| H(f_t(x_2)) \, \varrho(\dint t).
\end{align*}
Using H\"older's inequality, Jensen's inequality and $(r+s)^{1/4}\leq r^{1/4}+s^{1/4}$, $r,s\geq 0$, we obtain, for $\lambda^2$-a.e. $(x_1,x_2)\in\BX^2$,
\begin{align*}
\BE (D_{x_1}F)^4 & \leq \int \BE \prod_{i=1}^4  \big(h(X_{t_i}) |f_{t_i}(x_1)| + H(f_{t_i}(x_1))\big)  \, \varrho^4(\dint(t_1,t_2,t_3,t_4)) \\
& \leq \int \prod_{i=1}^4 \bigg[\BE \big(h(X_{t_i}) |f_{t_i}(x_1)| + H(f_{t_i}(x_1))\big)^4\bigg]^{1/4}  \, \varrho^4(\dint(t_1,t_2,t_3,t_4)) \allowdisplaybreaks\\
& \leq 8 \int \prod_{i=1}^4 \bigg[\BE h(X_{t_i})^4 |f_{t_i}(x_1)|^4 + H(f_{t_i}(x_1))^4\bigg]^{1/4}  \, \varrho^4(\dint(t_1,t_2,t_3,t_4)) \allowdisplaybreaks\\
& \leq 8 \int \prod_{i=1}^4 \big( C_2^{1/4} |f_{t_i}(x_1)| + H(f_{t_i}(x_1)) \big)  \, \varrho^4(\dint(t_1,t_2,t_3,t_4)) \allowdisplaybreaks\\
& \leq 8 \bigg( \int  C_2^{1/4} |f_{t}(x_1)| + H(f_{t}(x_1)) \, \varrho(\dint t) \bigg)^4 \\
& \leq 8  \bigg( \int  \psi_t(x_1) \, \varrho(\dint t) \bigg)^4
\end{align*}
and
\begin{align*}
\BE (D^2_{x_1,x_2}F)^4
 & \leq  \int \BE \prod_{i=1}^4 \big(h(X_{t_i}) |f_{t_i}(x_1)| \, |f_{t_i}(x_2)| + |f_{t_i}(x_2)| H(f_{t_i}(x_1))\\
 & \hskip 2.5cm +|f_{t_i}(x_1)| H(f_{t_i}(x_2))\big) \varrho^4(\dint(t_1,t_2,t_3,t_4))\\
 & \leq 27 \int \prod_{i=1}^4 \big(C_2^{1/4} |f_{t_i}(x_1)| \, |f_{t_i}(x_2)| + |f_{t_i}(x_2)| H(f_{t_i}(x_1)) +|f_{t_i}(x_1)| H(f_{t_i}(x_2))\big) \\
 & \hskip 2cm \varrho^4(\dint(t_1,t_2,t_3,t_4))\\
 & \leq 27 \bigg(\int \psi_t(x_1) \psi_t(x_2)\, \varrho(\dint t) \bigg)^4.
\end{align*}
We aim at applying Theorem \ref{thm:mainKolmogorov} to $(F-\BE F)/\sqrt{\BV F}$. First of all, we exploit Lemma \ref{lem:BoundFourthMoment} to show that
\begin{align*}
\frac{\BE (F-\BE F)^4}{(\BV F)^2} &\leq \max\bigg\{\frac{256}{(\BV F)^2} \bigg[ \int \big[\BE(D_zF)^4\big]^{1/2} \, \lambda(\dint z)\bigg]^2, \frac{4}{(\BV F)^2}\int \BE(D_zF)^4 \, \lambda(\dint z) +2 \bigg\}\\
& \leq \max\bigg\{ 2048 \bigg[ \iint \psi_{t_1}(x) \psi_{t_2}(x) \, \varrho^2(\dint(t_1,t_2))\, \lambda(\dint x) \bigg]^2/(\BV F)^2,\\
& \hskip 1cm  4 \iint \psi_{t_1}(x)\psi_{t_2}(x) \psi_{t_3}(x)\psi_{t_4}(x) \, \varrho^4(\dint(t_1,t_2,t_3,t_4)) \, \lambda(\dint x)/(\BV F)^2+2 \bigg\}\allowdisplaybreaks\\
& \leq \max\bigg\{ 2048 \int \langle \psi_{t_1}, \psi_{t_2} \rangle \langle \psi_{t_3}, \psi_{t_4} \rangle \, \varrho^4(\dint(t_1,t_2,t_3,t_4))/(\BV F)^2,\\
& \hskip 1cm  4 \int \langle \psi_{t_1}\psi_{t_2}, \psi_{t_3}\psi_{t_4}\rangle \, \varrho^4(\dint(t_1,t_2,t_3,t_4))/(\BV F)^2+2 \bigg\}<\infty,
\end{align*}
showing, in particular, that $F\in {\rm dom} D$. Using the Cauchy-Schwarz inequality, we obtain that
\begin{align*}
\gamma_1^2 & \leq \frac{96 \sqrt{6}}{(\BV F)^2} \iint \psi_{t_1}(x_1) \psi_{t_2}(x_2) \psi_{t_3}(x_1)\\
 & \hskip 3cm \psi_{t_3}(x_3) \psi_{t_4}(x_2) \psi_{t_4}(x_3) \, \varrho^4(\dint(t_1,t_2,t_3,t_4)) \, \lambda^3(\dint(x_1,x_2,x_3))\\
& = \frac{96 \sqrt{6}}{(\BV F)^2} \int \langle \psi_{t_1}, \psi_{t_3}\rangle  \langle \psi_{t_2},\psi_{t_4} \rangle  \langle \psi_{t_3},\psi_{t_4} \rangle \, \varrho^4(\dint(t_1,t_2,t_3,t_4)).
\end{align*}
Analogously, we have
\begin{align*}
\gamma_2^2 & \leq \frac{27}{(\BV F)^2} \iint \psi_{t_1}(x_1) \psi_{t_1}(x_3) \psi_{t_2}(x_1) \psi_{t_2}(x_3)\\
 & \hskip 3cm \psi_{t_3}(x_2) \psi_{t_3}(x_3) \psi_{t_4}(x_2) \psi_{t_4}(x_3)\, \varrho^4(\dint(t_1,t_2,t_3,t_4)) \, \lambda^3(\dint(x_1,x_2,x_3))\\
 & = \frac{27}{(\BV F)^2} \int \langle \psi_{t_1},\psi_{t_2} \rangle \langle \psi_{t_3}, \psi_{t_4} \rangle  \langle \psi_{t_1}\psi_{t_2} , \psi_{t_3}\psi_{t_4} \rangle \, \varrho^4(\dint(t_1,t_2,t_3,t_4)), \allowdisplaybreaks\\
 \gamma_3 & \leq \frac{8^{3/4}}{(\BV F)^{3/2}} \iint \psi_{t_1}(x_1) \psi_{t_2}(x_1) \psi_{t_3}(x_1) \varrho^3(\dint(t_1,t_2,t_3)) \, \lambda(\dint x_1) \\
  & = \frac{8^{3/4}}{(\BV F)^{3/2}} \int  \langle \psi_{t_1} \psi_{t_2}, \psi_{t_3}\rangle \, \varrho^3(\dint(t_1,t_2,t_3)), \allowdisplaybreaks \\
  \gamma_4 & \leq \frac{[\BE (F-\BE F)^4]^{1/4}}{2(\BV F)^2} \int  \langle \psi_{t_1} \psi_{t_2}, \psi_{t_3}\rangle \, \varrho^3(\dint(t_1,t_2,t_3)), \allowdisplaybreaks \\
  \gamma_5^2 & \leq \frac{8}{(\BV F)^2} \iint  \psi_{t_1}(x_1)  \psi_{t_2}(x_1)  \psi_{t_3}(x_1)  \psi_{t_4}(x_1) \, \varrho^4(\dint(t_1,t_2,t_3,t_4)) \, \lambda(\dint x_1)\\
  & = \frac{8}{(\BV F)^2} \int  \langle \psi_{t_1}\psi_{t_2}, \psi_{t_3} \psi_{t_4} \rangle \varrho^4(\dint(t_1,t_2,t_3,t_4)), \allowdisplaybreaks \\
  \gamma_6^2 & \leq \frac{36 \sqrt{6}}{(\BV F)^2}  \iint \psi_{t_1}(x_1) \psi_{t_2}(x_1)  \psi_{t_3}(x_1)\\
   & \hskip 3cm \psi_{t_3}(x_2) \psi_{t_4}(x_1) \psi_{t_4}(x_2) \, \varrho^4(\dint(t_1,t_2,t_3,t_4)) \, \lambda^2(\dint(x_1,x_2)) \allowdisplaybreaks\\
  & \quad +\frac{81}{(\BV F)^2} \int \psi_{t_1}(x_1)\psi_{t_1}(x_2) \psi_{t_2}(x_1)\psi_{t_2}(x_2) \psi_{t_3}(x_1)\psi_{t_3}(x_2) \psi_{t_4}(x_1)\psi_{t_4}(x_2) \\
  & \hskip 3cm  \varrho^4(\dint(t_1,t_2,t_3,t_4)) \, \lambda^2(\dint(x_1,x_2)) \\
  & \leq \frac{36 \sqrt{6}}{(\BV F)^2}  \int \langle \psi_{t_1} \psi_{t_2}, \psi_{t_3} \psi_{t_4} \rangle \langle \psi_{t_3},\psi_{t_4} \rangle \, \varrho^4(\dint(t_1,t_2,t_3,t_4))\\
  & \quad +\frac{81}{(\BV F)^2} \int \langle \psi_{t_1} \psi_{t_2},\psi_{t_3}\psi_{t_4}\rangle \langle \psi_{t_1}\psi_{t_2},\psi_{t_3}\psi_{t_4}\rangle \, \varrho^4(\dint(t_1,t_2,t_3,t_4)).
\end{align*}

\end{proof}

\subsection{Moving averages}\label{ss:ma}

We illustrate Theorem \ref{thm:PoissonShotNoise} by focussing on stationary random fields $(X_t)_{t\in\R^d}$ defined in the following way. We let $\eta$ be a Poisson process on $\R \times \R^d$ with intensity measure $\lambda(\dint u, \dint x)=\nu(\dint u) \, \dint x$, where the measure $\nu$ on $\R$ satisfies
$$
\int |u|^j \, \nu(\dint u)<\infty, \quad j=1,2.
$$
We further let $f:\R^d \to \R$ be such that
$$
\int |f(x)|+f(x)^2 \, \dint x<\infty
$$
and define
$$
X_t = \int u f(t-x) \, \hat{\eta}(\dint(u,x)), \quad t\in\R^d,
$$
where $\hat{\eta}(\dint u, \dint x)=\eta(\dint u, \dint x)-\lambda(\dint u,\dint x)$. This means that $X_t=I_1(f_t)$, where $f_t(u,x):=u f(t-x)$. We are interested in the normal approximation of the random variables
$$
F_T := \int_{W_T} \varphi(X_t) \, \dint t, \quad T>0,
$$
where $W_T:=T^{1/d} [0,1]^d$ is a cube with volume $T$.

\begin{theorem}\label{thm:stationary}
Let $(X_t)_{t\in\R^d}$ and $F_T$ be as above and let $N$ be a standard Gaussian random variable. Assume that
\begin{equation}\label{eq:VariancePoissonShotNoise}
\BV F_T \geq \sigma T, \quad T\geq t_0,
\end{equation}
with $\sigma,t_0>0$ and that there are finite constants $p,\tilde{C}>0$ such that
\begin{equation}\label{eq:Boundhi}
|\varphi(r)| + |\varphi'(r)| + |\varphi''(r)| \leq \tilde{C} \big(1+|r|^p\big), \quad r\in\R,
\end{equation}
and
$$
m:=16\int u^2 \, \nu(\dint u)+16\int |u|^{4+4p} \, \nu(\dint u)<\infty \quad \text{and} \quad C_2:=\max\{\BE |X_0|^{4+4p},1\}<\infty.
$$
Moreover, assume that the function $g:\R^d\to \R$ given by
$$
g(y)=\big(C_2^{1/4}+\tilde{C}\big) \big(|f(y)|+3^p|f(y)|^{1+p}\big)
$$
satisfies
$$
M:=\max\bigg\{\int g(z) \, \dint z, \int \bigg(\int g(y-z) g(y)\, \dint y\bigg)^4 \, \dint z \bigg\}< \infty.
$$
Then, there is a finite constant $C>0$ depending on $\sigma$, $m$, $M$ and $t_0$ such that
$$
d_K\bigg(\frac{F_T-\BE F_T}{\sqrt{\BV F_T}}, N\bigg) \leq \frac{C}{\sqrt{T}}, \quad T\geq t_0.
$$
\end{theorem}

\begin{remark}{\rm Standard computations show that the assumption $M<\infty$ is satisfied in the following two standard cases: (i) $f$ is a bounded function with compact support, and (ii) $f(x) = c\exp(-\langle v, x\rangle){\bf 1}\{x_i\geq 0, \,\, i=1,...,d\}$, where $c\in \R$ and $v = (v_1,...,v_d)\in\R^d$ with $v_i>0$ for $i=1,...,d$. In the case $d=1$ the process $X_t$ is called an {\em Ornstein-Uhlenbeck L\'evy process}. See \cite{BNS} and \cite{DBPP, PePrU}, respectively, for applications of these processes in mathematical finance and Bayesian statistics. See \cite{PSTU10} for a number of related CLTs involving linear and quadratic functionals of Ornstein-Uhlenbeck L\'evy processes.

}
\end{remark}

\begin{proof}
It follows from \eqref{eq:Boundhi} that the assumptions \eqref{eq:boundphi0}, \eqref{eq:boundphi1} and \eqref{eq:boundphi2} are satisfied with
$$
h(r)= \tilde{C} \big(1+3^p |r|^p\big), \quad r\in\R.
$$
Together with the special structure of $f_t$ we see that
\begin{align*}
\psi_t(u,x) & \leq C_2^{1/4}|f_t(u,x)|+ \tilde{C} (1+3^p |f_t(u,x)|^p) |f_t(u,x)|\\
& \leq \big(C_2^{1/4}+\tilde{C} \big) (|u|+|u|^{1+p}) \big(|f(t-x)|+3^p|f(t-x)|^{1+p}\big)\\
& = (|u|+|u|^{1+p}) \, g(t-x).
\end{align*}
Using $\psi_t(u,x)=\psi_{t-s}(u,x-s)$, $u\in\R$, $s,t,x\in\R^d$, this estimate and the product form of $\lambda$, we obtain that
\begin{align*}
& \int_{W_T^4} \langle \psi_{t_1}, \psi_{t_3} \rangle \langle \psi_{t_2}, \psi_{t_4} \rangle \langle \psi_{t_3} ,\psi_{t_4}\rangle  \, \dint(t_1,t_2,t_3,t_4)\\
& \leq \int_{(\R^d)^3}\int_{W_T}
\langle \psi_{t_1-t_4}, \psi_{t_3-t_4} \rangle \langle \psi_{t_2-t_4}, \psi_{0} \rangle
\langle \psi_{t_3-t_4},\psi_{0}\rangle  \,\dint t_4 \,\dint(t_2,t_3,t_4) \allowdisplaybreaks\\
& = T \int \langle \psi_{t_1}, \psi_{t_3} \rangle \langle \psi_{t_2}, \psi_{0} \rangle \langle \psi_{t_3}, \psi_{0}\rangle  \, \dint(t_1,t_2,t_3) \allowdisplaybreaks\\
& \leq m^3 \, T \int g(t_1-x_1) g(t_3-x_1) g(t_2-x_2) g(-x_2) g(t_3-x_3) g(-x_3)  \, \dint(x_1,x_2,x_3,t_1,t_2,t_3) \\
& = m^3 \, T \bigg( \int g(z) \, \dint z\bigg)^6.
\end{align*}
In a similar way, we deduce that
\begin{align*}
& \int_{W_T^4} \langle \psi_{t_1},\psi_{t_2}\rangle \langle \psi_{t_3},\psi_{t_4}\rangle  \langle \psi_{t_1}\psi_{t_2},\psi_{t_3}\psi_{t_4}\rangle \, \dint(t_1,t_2,t_3,t_4)\\
& \leq T \int \langle \psi_{t_1},\psi_{t_2}\rangle \langle \psi_{t_3},\psi_{0}\rangle  \langle \psi_{t_1}\psi_{t_2},\psi_{t_3}\psi_{0}\rangle \, \dint(t_1,t_2,t_3)\\
& \leq m^3 \, T \int g(t_1-x_1)g(t_2-x_1) g(t_3-x_2)g(-x_2)\\
& \hskip 2cm g(t_1-x_3)g(t_2-x_3)g(t_3-x_3) g(-x_3) \, \dint(x_1,x_2,x_3,t_1,t_2,t_3)\\
& \leq m^3 \, T \int g(t_1) g(t_2) g(t_3) g(-x_2) g(t_1+x_1-x_3)\\
& \hskip 2cm g(t_2+x_1-x_3) g(t_3+x_2-x_3) g(-x_2+x_2-x_3) \, \dint(x_1,x_2,x_3,t_1,t_2,t_3)\\
& = m^3 \, T \bigg(\int \bigg(\int g(y-z) g(y)\, \dint y\bigg)^2 \, \dint z \bigg)^2,
\end{align*}
and then also
\begin{align*}
& \int_{W_T^4} \langle \psi_{t_1}\psi_{t_2},\psi_{t_3}\psi_{t_4}\rangle \langle \psi_{t_3},\psi_{t_4}\rangle \, \dint(t_1,t_2,t_3,t_4)\\
& \leq m^2 \, T \bigg(\int g(z) \, \dint z\bigg)^2 \int \bigg(\int g(y-z)g(y) \, \dint y\bigg)^2 \, \dint z,
\end{align*}
$$
 \int_{W_T^4} \langle \psi_{t_1}\psi_{t_2},\psi_{t_3}\psi_{t_4}\rangle^2 \, \dint(t_1,t_2,t_3,t_4) \leq m^2 \, T \int \bigg(\int g(y-z) g(y) \, \dint y\bigg)^4 \, \dint z,
$$
$$
\int_{W_T^4} \langle \psi_{t_1},\psi_{t_2}\rangle \langle \psi_{t_3},\psi_{t_4}\rangle \, \dint(t_1,t_2,t_3,t_4) \leq m^2 \, T^2  \bigg(\int g(z) \, \dint z\bigg)^4,
$$
$$
\int_{W_T^4} \langle \psi_{t_1}\psi_{t_2},\psi_{t_3}\psi_{t_4}\rangle \, \dint(t_1,t_2,t_3,t_4)
\leq m \, T \bigg(\int g(z) \, \dint z\bigg)^4
$$
and
$$
\int_{W_T^3} \langle \psi_{t_1}\psi_{t_2},\psi_{t_3} \rangle \, \dint(t_1,t_2,t_3) \leq m \, T \bigg(\int g(z) \, \dint z\bigg)^3.
$$
Now the assertion follows from Theorem \ref{thm:PoissonShotNoise}.
\end{proof}

The following lemma helps to check assumption \eqref{eq:VariancePoissonShotNoise}.

\begin{lemma}\label{l8.4}
Let the assumptions of Theorem \ref{thm:stationary} hold and let
$$
\tilde{\varphi}(r)=\BE[\varphi(X_0+r)-\varphi(X_0)], \quad r\in\R.
$$
Then,
$$
\liminf_{T\to\infty}\frac{\BV F_T}{T} \geq \int \bigg(\int \tilde{\varphi}(uf(t)) \, \dint t\bigg)^2 \, \nu(\dint u)=:\sigma
$$
and \eqref{eq:VariancePoissonShotNoise} is satisfied whenever $\sigma\in (0,\infty)$.
\end{lemma}

\begin{proof}
We have that, for $u\in\R$ and $x\in\R^d$,
$$
\BE D_{(u,x)}F_T = \BE \int_{W_T} \varphi(X_t+uf(t-x)) - \varphi(X_t) \, \dint t = \int_{W_T} \tilde{\varphi}(uf(t-x)) \, \dint t,
$$
where we have used the stationarity of $(X_t)_{t\in \R^d}$. Together with \eqref{covcha}, we obtain that
\begin{align*}
\frac{\BV F_T}{T} & \geq \frac{1}{T} \iint \I\{t_1,t_2\in W_T\} \tilde{\varphi}(uf(t_1-x)) \tilde{\varphi}(uf(t_2-x)) \, \dint(t_1,t_2,x) \, \nu(\dint u)\\
& = \iint \tilde{\varphi}(uf(t_1)) \tilde{\varphi}(uf(t_2)) \frac{\ell_d(W_T\cap (W_T+t_1-t_2))}{T}\, \dint(t_1,t_2) \, \nu(\dint u).
\end{align*}
By the assumptions of Theorem \ref{thm:stationary}, we see that
\begin{align*}
& \iint |\tilde{\varphi}(uf(t_1)) \tilde{\varphi}(uf(t_2))| \, \dint(t_1,t_2) \, \nu(\dint u)\\
& \leq \tilde{C}^2 \iint (1+2^{p-1}\BE|X_0|^p+2^{p-1}|u|^p |f(t_1)|^p) \, |u| \, |f(t_1)|\\
& \hskip 1.8cm (1+2^{p-1}\BE|X_0|^p+2^{p-1}|u|^p |f(t_2)|^p) \, |u| \, |f(t_2)| \, \dint(t_1,t_2) \, \nu(\dint u)<\infty.
\end{align*}
Now the dominated convergence theorem concludes the proof.
\end{proof}

In the special case where $\varphi$ is strictly increasing, $f\geq 0$ and $\nu$ is not concentrated at the origin, we have that
$$
\int \tilde{\varphi}(uf(t)) \, \dint t \neq 0
$$
for all $u\neq 0$ so that the previous lemma implies that $\BV F_T\geq \sigma T$, $t\geq t_0$, with constants $\sigma,t_0>0$. For the Ornstein-Uhlenbeck process discussed in the introduction (see Proposition \ref{introexOU}), we have that
$$
\int_{-\infty}^{\infty} \tilde\varphi(u\I\{t\geq 0\} \exp(-t)) \, \dint t = \int_0^{\infty} \tilde\varphi(u\exp(-t)) \, \dint t = \int_0^u \frac{1}{r} \, \tilde\varphi(r) \, \dint r.
$$
Consequently, the assumption \eqref{eq:VariancePoissonShotNoise} is satisfied if $\int_0^u \tilde{\varphi}(r) \, \dint r \neq 0$ for some $u$ in the support of $\nu$.

\bigskip

\noindent{\bf Acknowledgments.} GP wishes to thank Rapha\"el Lachi\`eze-Rey for useful discussions.

\end{document}